\documentclass[12pt]{article}
\usepackage{graphicx}
\usepackage{color}
\usepackage[colorlinks, linkcolor=blue,  citecolor=blue, urlcolor=blue]{hyperref}%
\usepackage{url}
\usepackage{subfig}
\usepackage{amsmath,amsthm}
\usepackage{amssymb,enumerate}
\usepackage[T1]{fontenc}
\theoremstyle{plain}
\newtheorem{theorem}{Theorem}[section]
\newtheorem{lemma}[theorem]{Lemma}
\newtheorem{prop}[theorem]{Proposition}

\theoremstyle{definition}
\newtheorem{definition}{Definition}[section]

\theoremstyle{remark}
\newtheorem*{rem}{Remark}

\begin{document}
\title{Explicit Flock Solutions for Quasi-Morse potentials}
\author{J. A. CARRILLO,
    Y. HUANG,
    S. MARTIN \\
Department of Mathematics, Imperial College
    London, \\
 London, SW7 2AZ, UK\footnote{ 
    Email\textup{\nocorr: \{carrillo,
    yanghong.huang, stephan.martin\}@imperial.ac.uk }
}
}
\maketitle
\begin{abstract}
We consider interacting particle systems and their mean-field
limits, which are frequently used to model collective aggregation
and are known to demonstrate a rich variety of pattern formations.
The interaction is based on a pairwise potential combining
short-range repulsion and long-range attraction. We study
particular solutions, that are referred to as flocks in the
second-order models, for the specific choice of the Quasi-Morse
interaction potential. Our main result is a rigorous analysis of
continuous, compactly supported flock profiles for the
biologically relevant parameter regime. Existence and uniqueness
are proven for three space dimensions, whilst existence is shown for
the two-dimensional case. Furthermore, we numerically investigate
additional Morse-like interactions to complete the understanding
of this class of potentials.
\end{abstract}

\section{Introduction}
Self-organization, complex pattern formation, and rich dynamic
structures are common features of collective motion of
individuals. Fish shoals, bird flocks, insects swarms,
myxobacteria formations, and many others are just particular
instances of these fascinating phenomena
\cite{Camazine_etal,Couzin:Krause}. A large number of models have
been introduced based on social interaction mechanisms between
individuals, namely: long-range attraction, short-range repulsion,
and alignment; see \cite{Huth:Wissel1992,HCH,lukeman} for example.

Here, we concentrate on the by-now classical models in which the
attraction and repulsion between individuals are taken into account
via a pairwise radial potential $W(x)=U(|x|)$. A first-order
aggregation model of swarming
(\cite{M&K,nano1,nano2,BertozziCarilloLaurent}) then reads
\begin{equation}
\label{eq:aggr}
 \frac{dx_i}{dt} = - \frac{1}{N} \sum_{j\neq i} \nabla W(x_i-x_j).
\end{equation}
For a second-order model for swarming, an
asymptotic cruise speed is fixed by the balance of self-propulsion
and friction terms, see \cite{PhysRevE.63.017101,d2006self}. The governing system of equations
for the particle dynamics $(x_i,v_i) \in \mathbb{R}^n\times
\mathbb{R}^n, i=1,2,\ldots,N$ is
\begin{equation}
\label{eq:partsys2}
    \begin{split}
 \frac{dx_i}{dt} &= v_i, \\
\frac{dv_i}{dt} &= \alpha v_i-\beta |v_i|^2 v_i - \frac{1}{N}
\sum_{j\neq i} \nabla W(x_i-x_j).
    \end{split}
\end{equation}
The self-propulsion term $\alpha v_i - \beta |v_i|^2 v_i$ with
Rayleigh-type dissipation can also be generalized to the form
$f(|v_i|)v_i$ for some function $f: [0,\infty) \to \mathbb{R}$,
such that $f(0)>0$ and $f(\upsilon)$ becomes negative when
$\upsilon$ is large enough. In both models, the potential $W$ is
assumed to be repulsive at short range ($U(r)$ decreases for small
$r>0$) and attractive at long range ($U(r)$ increases for $r$
large enough). The most popular one used in the literature is the
Morse-type potential~\cite{PhysRevE.63.017101,d2006self}:
\begin{equation}\label{eq:MorsePotential}
    U(r) = C_{\mathcal{R}} e^{-r/\ell_\mathcal{R}} - C_\mathcal{A} e^{-r/\ell_\mathcal{A}},
\end{equation}
where $C_\mathcal{R},C_\mathcal{A}$ specify the strength of the
repulsive and attractive forces, and
$\ell_\mathcal{R},\ell_\mathcal{A}$ specify their length scales.

Depending on the parameters, the system~\eqref{eq:partsys2}
exhibits a rich variety of patterns: flocks, rotating mills,
rings, and clumps~\cite{PhysRevE.63.017101,d2006self}. To further
study the emergence and bifurcation of these patterns, one has to
resort to the corresponding continuum equations, derived from
either kinetic theory or mean field approximation in the limit
when the number of particles $N$ goes to infinity. The system of
equations for the continuous density $\rho$ and the velocity $u$
reads ~\cite{PhysRevE.63.017101,MR2369988,MR2507454}
\begin{equation}
\begin{split}
    \frac{\partial \rho}{\partial t} + \operatorname{div} (\rho u)
    &=0,\cr
    \frac{\partial u}{\partial t} + (u\cdot \nabla )u &=
    (\alpha - \beta |u|^2)u - \nabla W\star \rho,
\label{eq:conteq}
\end{split}
\end{equation}
where $W\star\rho$ is the convolution between $W$ and $\rho$. 
In particular, a coherent moving flock is a solution such that
$u(x,t)=u_0$, $\rho(x,t) = \rho_F(x-u_0t)$ for some constant
velocity $u_0$ with $|u_0|^2=\tfrac{\alpha}{\beta}$, and steady
density $\rho_F$ satisfying the equation $\nabla W\star \rho_F =
0$ on the support of
$\rho_F$~\cite{MR2507454,CKMT,Carrillo2013,ABCV}. If we deal with
densities supported on an open set, the existence of flock
solutions for~\eqref{eq:partsys2} is reduced to $W\star \rho = D$,
on $\text{supp}[\rho]$ for some constant $D$, where the subscript
$F$ for the steady flock solution $\rho_F$ is dropped in the rest
of the paper for simplicity.

As a matter of fact, flock solutions in this generality coincide with the
stationary solutions for the first-order continuum model derived from
\eqref{eq:aggr}, which reads
\begin{equation}
    \frac{\partial \rho}{\partial t} + \operatorname{div} (( -\nabla W\star \rho)\rho )
    = 0.
    \label{eq:contaggr}
\end{equation}

The existence of some particular explicit stationary solutions
where the density is uniformly concentrated on a ring
\cite{KSUB,BCLR}, both for the discrete model \eqref{eq:aggr} and
the continuum case \eqref{eq:contaggr}, has led to a thorough
study of their stability and properties in the framework of the
first-order models
\cite{KSUB,predict,soccerball,BCLR,BCLR2,bigring}. The stability
of the ring flock solutions for the second-order model
\eqref{eq:partsys2} has been recently tackled in \cite{ABCV}.
However, in many instances, as in the archetypical Morse
potentials, we do observe nicely compactly supported radial flocks
in simulations. In the rest of this work, we will concentrate in
finding non-concentrated flock profiles for both \eqref{eq:conteq}
and \eqref{eq:contaggr}:

\begin{definition}[Flock profile]
For a given $W$, a \emph{flock profile} is defined as a radially
symmetric continuous probability density $\rho(r)$, compactly supported on a
ball of radius $R_F$ satisfying the characteristic equation
\begin{equation}\label{eq:steadyeq}
 W\star \rho = D,\qquad  \text{on } \text{supp}[\rho]=B(0,R_F)
 \mbox{ for some constant $D$}.
\end{equation}
\end{definition}

Despite their observation in simulations of~\eqref{eq:partsys2}
with a variety of attractive-repulsive potentials, there is nearly
no analytical study of the existence and bifurcation of these
flocks in the parameter space. The reason lies in the great
difficulties in solving the integral equation \eqref{eq:steadyeq}
for popular potentials like~\eqref{eq:MorsePotential}. Multiple
solutions may exist (see ~\cite{PhysRevE.63.017101}) by a Newton
solver, where the non-physical solutions are shown to be unstable.
Other solutions that are available are in general asymptotic, when
the the density is concentrated on a thin
annulus~\cite{MR2788924}. Another fully explicit case corresponds
to the Newtonian repulsion with quadratic confinement
$W(x)=\tfrac{|x|^2}2-\tfrac{|x|^{2-n}}{2-n}$ for which the
solution is the characteristic function of a ball with suitable
radius. However, for any other member of the family of potentials 
$$
W(x)=\frac{|x|^a}a-\frac{|x|^b}b, \qquad a>b\geq 2-n\,,
$$
with the convention that $\tfrac{|x|^0}0=\log x$, they are no
longer explicit, see \cite{FHK,FH,BCLR2}. Moreover, flock profiles
play an important role on the dynamics of \eqref{eq:partsys2}
since they form a stable family of attracting solutions as shown
in \cite{CHM2} for general potentials under suitable conditions.

One approach to get explicit solutions of
equation~\eqref{eq:steadyeq} is to replace $W$ with an
analytically more tractable kernel, for instance the so called
\emph{Quasi-Morse potential} proposed in~\cite{Carrillo2013},
instead of~\eqref{eq:MorsePotential}. The great simplification
with Quasi-Morse potential comes from an explicit expression
of $\rho$, characterized by only  three parameters, which is obtained by solving an
ODE derived from~\eqref{eq:steadyeq}. The three parameters are
found in \cite{Carrillo2013} by a numerical procedure involving
the computation of the convolution in the left-hand side
of~\eqref{eq:steadyeq}. The resulting numerical solutions in two
and three dimensions agree very well with those approximated from
the particle simulations. In this paper, we show that this
computationally intensive convolution can be evaluated as a few
algebraic terms, hence the existence/non-existence of the flock
profile in the parameter space can be discussed in
detail.

We start in Section 2 by summarizing the properties of the
Quasi-Morse potentials and deriving new explicit formulas for the
convolution \eqref{eq:steadyeq}. Section 3 is devoted to the
analysis of existence and uniqueness of flock profiles in the
three dimensional case, with respect to the parameter space of the
potential. In Section 4, we perform a similar analysis in two
dimensions to identify the existence of flock profiles in
parameter space. Due to the simplification of the Bessel
functions in three dimensions, the expressions are easier to
manage and the result obtained is more complete in three
dimensions. Section 5 deals with further remarks on the
Quasi-Morse potentials and asymptotic cases. Finally, we end this
work in Section 6 by investigating similar properties in Morse-like potentials
to numerically ascertain how generic the case of the Quasi-Morse
potential is.


\section{The Quasi-Morse potential and explicit flock profiles in
general dimensions} \label{sec2}

For completeness, we first review the basic properties and the
explicit solutions proposed in~\cite{Carrillo2013}. The new
pairwise \emph{Quasi-Morse potential} $W(x)=U(|x|)$ still assumes
the form $U(r)=V(r)-V_\ell(r)$, where now $V(r)$ is the
fundamental solution of the second-order differential operator
$\Delta -k^2\operatorname{Id}$ (i.e., $\Delta V - k^2 V =
\delta_0$) and $V_\ell(r) = CV(r/\ell)$ is a rescaled version of
$V(r)$ (i.e., $\Delta V_\ell -
\frac{k^2}{\ell^2}V_\ell=\ell^{n-2}\delta_0$). For simplicity,
here the attraction strength $C_\mathcal{A}$ and length scale
$\ell_\mathcal{A}$ are normalized to be unity, and then
$C=C_\mathcal{R}$ and $\ell=\ell_\mathcal{R}$.

The biologically relevant cases correspond to the radial potential
$U(r)$ possessing a unique global minimum at some positive radius.
It was proven in \cite{Carrillo2013} that the biologically
relevant parameter region is $C \ell^{n-2}>1$ and $\ell<1$ for
dimensions one to three . The explicit expressions for $V(r)$ in
these dimensions are given in~\cite{Carrillo2013} as
$-e^{-kr}/2k$, $-K_0(kr)/2\pi$, $-e^{-kr}/4\pi r$ respectively. To
present the discussion in a unified context for dimension $n$, we
write $V(r)$ in terms of the modified Bessel functions of the
second kind~\cite{MR1817225}, i.e.,
\[
 V(r) = -(2\pi)^{-\frac{n}{2}}r^{1-\frac{n}{2}}k^{\frac{n}{2}-1}
K_{\frac{n}{2}-1}(kr),
\]
and correspondingly
\begin{equation}\label{eq:quasimorse}
 U(r) =  (2\pi)^{-\frac{n}{2}}r^{1-\frac{n}{2}}k^{\frac{n}{2}-1}
\Big(
C\ell^{\frac{n}{2}-1}K_{\frac{n}{2}-1}\big(kr/\ell\big)-
K_{\frac{n}{2}-1}\big(kr\big)
\Big).
\end{equation}
In particular, $U$ reduces to the conventional Morse
potential~\eqref{eq:MorsePotential} in dimension one as
$K_{-\frac{1}{2}}(x)=\sqrt{\frac{\pi}{2x}}e^{-x}$ (see
Appendix~\ref{sec:bessel}, with other properties of the Bessel
function $J_\nu(x)$ and modified Bessel functions $K_\nu(x)$ and
$I_\nu(x)$ used later).

One of the advantages of the Quasi-Morse
potential~\eqref{eq:quasimorse} is that the integral
equation~\eqref{eq:steadyeq} can be transformed into a second-order ODE for the radial density $\rho(r)$. Applying the operators
$\Delta - k^2\operatorname{Id}$ and $\Delta -
\frac{k^2}{\ell^2}\operatorname{Id}$ to both sides
of~\eqref{eq:steadyeq} as in \cite{MR2788924,Carrillo2013}, the
density $\rho$ now satisfies
\begin{equation*}
    \Delta \rho +A \rho =
    \frac{k^4}{\ell^2-C\ell^{n}}D, \qquad
\text{ on}\ \text{ supp } \rho,
\end{equation*}
with the \emph{aggregate potential parameter} $A = k^2\big(1-C\ell^n\big)/\big(C\ell^n-\ell^2\big)$. In radial
coordinate $r$,  this equation reads
\begin{equation}\label{eq:rhoradialeq}
    \frac{1}{r^{n-1}}\frac{d}{dr}r^{n-1}\frac{d\rho}{dr}
    \pm a^2 \rho = \frac{k^4}{\ell^2-C\ell^{n}}D,\qquad
a = \sqrt{|A|}.
\end{equation}
The general solution, assumed to be bounded at the origin,
takes the form (see~\cite{Carrillo2013} for $n=2,3$)
\begin{equation}\label{eq:explicitsol}
     \rho(r) = \begin{cases}
        \mu_1 r^{1-\frac{n}{2}}J_{\frac{n}{2}-1}(ar)+\mu_2,\qquad & A>0,\cr
        \qquad\mu_1 r^2 +\mu_2, & A=0, \cr
        \mu_1 r^{1-\frac{n}{2}}I_{\frac{n}{2}-1}(ar)+\mu_2, & A<0,
    \end{cases}
\end{equation}
on $[0,R]$ and $\rho(r) \equiv 0$ on $(R,\infty)$. For any fixed
radius $R$, the parameters $\mu_1$ and $\mu_2$ have to be adjusted
to fit the integral equation~\eqref{eq:steadyeq} and ensure
positivity of $\rho(r)$ on $r \in [0,R]$. In fact, this is exactly
how the numerical solutions are obtained in~\cite{Carrillo2013},
where the observed flock profiles exist only when $A>0$. Despite
the perfect agreement with particle simulations, the convolution
$W\star\rho$ remains the bottleneck of the computation. In this
paper, we show that the convolution can also be reduced to a few
algebraic terms, eventually leading to  the rigorous
existence/non-existence proofs of radial solutions in the
different parameter regimes.

The simplification of the convolution $W\star\rho$ is suggested by
the following observation: when the operators $\Delta - k^2
\operatorname{Id}$ and $\Delta - \frac{k^2}{\ell^2}
\operatorname{Id}$ are applied to both sides
of~\eqref{eq:steadyeq}, we get a fourth-order ordinary
differential equation (in the radial coordinate $r$)
\[
    \left(
\frac{1}{r^{n-1}}\frac{d}{dr}r^{n-1}\frac{d}{dr}-\frac{k^2}{\ell^2}
    \right)
    \left(
\frac{1}{r^{n-1}}\frac{d}{dr}r^{n-1}\frac{d}{dr}-k^2
\right) W\star\rho = \frac{k^4}{\ell^2}D
\]
for the radial function $W\star \rho$, which is equivalent
to~\eqref{eq:rhoradialeq}. The general solution of the fourth-order ODE takes the form
\begin{align}\label{eq:Wrhosol}
    (W\star \rho)(r) =
    D \,&+ \lambda_1r^{1-\frac{n}{2}} I_{\frac{n}{2}-1}(kr/\ell)+
    \lambda_2   r^{1-\frac{n}{2}} I_{\frac{n}{2}-1}(kr)\nonumber\\&+ \lambda_3 r^{1-\frac{n}{2}} K_{\frac{n}{2}-1}(kr/\ell)+
    \lambda_4   r^{1-\frac{n}{2}} K_{\frac{n}{2}-1}(kr), \qquad 0\leq r \leq R,
\end{align}
for some coefficients $\lambda_1,\dots,\lambda_4$.
We will find the desired flock profiles when all $\lambda_i$ vanish and thus \eqref{eq:steadyeq} is fulfilled.
We first notice that $\lambda_3$ and $\lambda_4$ have to vanish in order to have a bounded solution at the origin with
bounded derivatives.
Imposing that  $\lambda_1$ and $\lambda_2$ vanish will lead to necessary and sufficient conditions for a flock profile.
Following this strategy, $D, \lambda_1$ and $\lambda_2$ will be expressed in terms of the support size $R$ and
the coefficients $\mu_1, \mu_2$ by inserting
\eqref{eq:explicitsol} into the left-hand side of
\eqref{eq:Wrhosol}.

First, we compute $\lambda_1, \lambda_2$ for the explicit
solution in \eqref{eq:explicitsol}. It turns out that the
convolution $W\star\rho$ can be obtained by direct integrations.
To start, because of the radial symmetry, $W\star \rho$ can be
written as
\begin{equation}\label{eq:convrho}
(W\star {\rho})(x) = \int_{|y|\leq R} W(x-y){\rho}(|y|)dy
 =\int_0^{R} \left(\int_{\partial B(0,1)} W(x-s\omega) d\omega
\right) {\rho}(s)s^{n-1}ds.
\end{equation}
This convolution, as a function of $r=|x|$,
simplifies in the particular case of the Quasi-Morse
potential $W(x) = V(|x|)-CV(|x|/\ell)$. In fact, the integral on
the unit sphere $\partial B(0,1)$ above can be evaluated using the
following formula (see~\cite[p. 90]{MR0232968})
\begin{multline}\label{eq:radsimp}
\int_0^\pi \big(a^2+b^2-2ab\cos\theta\big)^{-\nu/2}
K_\nu\Big( \big(a^2+b^2-2ab\cos\theta\big)^{1/2}\Big)\sin^{2\nu}\theta
d\theta \cr
=\pi^{1/2}\Gamma\Big(\frac{1}{2}+\nu\Big)
\Big(\frac{2}{ab}\Big)^{\nu}I_\nu\big(\min(a,b)\big)K_\nu\big(\max(a,b)\big).
\end{multline}
Let us detail the computation of this angular integral for the second
component $V_\ell(r) = CV(r/\ell)$ of $W$, as the integral for
$V(r)$ is the special case of $C=\ell=1$. Setting $\nu=n/2-1$,
$a=kr/\ell$ and $b=ks/\ell$, the angular integration involving
$V_\ell$ in \eqref{eq:convrho} reads
\begin{align}
 \int_{\partial B(0,1)} \!\!\!\!\!\!V_\ell(x-s\omega) d\omega
&=-C\frac{2\pi^{\frac{n-1}{2}}}{\Gamma\left(\frac{n-1}{2}\right)}
(2\pi)^{-\frac{n}{2}}k^{n-2} \int_0^\pi D(\theta)^{-\nu/2}
K_\nu\big(D(\theta)^{1/2}\big) \sin^{2\nu}\theta d\theta \nonumber
\\ &= - C\ell^{n-2}(rs)^{1-\frac{n}{2}}I_{\frac{n}{2}-1}\Big(
\frac{k}{\ell}\min(r,s)\Big) K_{\frac{n}{2}-1}\Big(
\frac{k}{\ell}\max(r,s)\Big), \label{eq:angint}
\end{align}
where $D(\theta)=\frac{k^2}{\ell^2}(r^2+s^2-2rs\cos\theta)$. As a
result, the convolution~\eqref{eq:convrho} becomes an integral in
$s$ only and the convolution of the repulsive potential $V_\ell$
with a density $\rho$ supported on the ball $B(0,R)$ is
\begin{multline}\label{eq:repconvint}
  V_\ell\star \rho(x) =  C\ell^{n-2}r^{1-\frac{n}{2}}
    \left[
        K_{\frac{n}{2}-1}(kr/\ell)\int_0^r
        s^{\frac{n}{2}}I_{\frac{n}{2}-1}(ks/\ell)\rho(s) ds
\right. \cr \left.        +
      I_{\frac{n}{2}-1}(kr/\ell)\int_r^R
        s^{\frac{n}{2}}K_{\frac{n}{2}-1}(ks/\ell)\rho(s) ds
        \right],
\end{multline}
for $0\leq r=|x|\leq R$. This integral, when $\rho$ takes the special
form~\eqref{eq:explicitsol}, can be further simplified using
various integral identities of (modified) Bessel functions. Since
these algebraic manipulations do not bring any further insights, we
have postponed them to Appendix~\ref{sec:expconvdev}. The final
result, whose general forms are already expected from~\eqref{eq:Wrhosol},
 is as follows.

\begin{prop}\label{prop1}
Given the Quasi-Morse potential $W(x)=U(|x|)$ in
\eqref{eq:quasimorse} and $\rho$ defined in
\eqref{eq:explicitsol}, the convolution $W\star \rho$ has the
expression:
\begin{equation}\label{eq:explicitconv}
W\star \rho (x) =
\begin{cases}
    \frac{\mu_2}{k^2}(C\ell^n-1)+\frac{R^{\frac{n}{2}}}{k}
    r^{1-\frac{n}{2}}\left[
        B_+(1)K_{\frac{n}{2}}(kR)I_{\frac{n}{2}-1}(kr)\right.\cr
\left.\qquad\qquad\quad
-C\ell^{n-1}B_+(\ell)K_{\frac{n}{2}}(kR/\ell)
        I_{\frac{n}{2}-1}(kr/\ell)
        \right]  \qquad & A>0,\\[2mm]
        \frac{2n\mu_1}{k^4}(C\ell^{n+2}-1)
        +R^{\frac{n}{2}}r^{1-\frac{n}{2}}\left[
            B_0(1)
            K_{\frac{n}{2}}(kR)
            I_{\frac{n}{2}-1}(kr)
            \right.
            \cr
\left.\qquad\qquad\quad  -C\ell^{n-1}B_0(\ell)
            K_{\frac{n}{2}}(kR/\ell)
I_{\frac{n}{2}-1}(kr/\ell)
            \right] \qquad &A=0,\\[2mm]
            \frac{\mu_2}{k^2}(C\ell^n-1)+\frac{R^{\frac{n}{2}}}{k}
            r^{1-\frac{n}{2}}\left[
                B_-(1)
                K_{\frac{n}{2}}(kR)
                I_{\frac{n}{2}-1}(kr)
                \right.\cr
                \left.\qquad\qquad\quad
                -C\ell^{n-1}B_-(\ell)K_{\frac{n}{2}}(kR/\ell)
                I_{\frac{n}{2}-1}(kr/\ell)
                \right] \qquad &A<0.
\end{cases}
    \end{equation}
    where $B_+(\xi)=\tilde{B}_+(\xi)\mu_1 + \mu_2,
    B_0(\xi) = \tilde{B}_0(\xi)\mu_1 +\mu_2,
    B_-(\xi) = \tilde{B}_-(\xi)\mu_1 + \mu_2$, and
\begin{flalign}
\tilde{B}_+(\xi) &=
    R^{1-\frac{n}{2}}\left(1+\frac{a^2\xi^2}{k^2}
    \right)^{-1}\left[ J_{\frac{n}{2}-1}(aR)
        \frac{K_{\frac{n}{2}-2}\big(kR/\xi\big)}
        {K_{\frac{n}{2}}\big(kR/\xi\big)} +
        \frac{a\xi}{k}J_{\frac{n}{2}-2}(aR)
        \frac{K_{\frac{n}{2}-1}\big(kR/\xi\big)}{
            K_{\frac{n}{2}}\big(kR/\xi\big)}\right],\nonumber\\
            \tilde{B}_0(\xi) &=
            \frac{2\xi}{k}R\frac{K_{\frac{n}{2}+1}\big(kR/\xi\big)}
            {K_{\frac{n}{2}}\big(kR/\xi\big) }
            +1,\label{Bequationsgeneral}\\
    \tilde{B}_-(\xi) &=
    R^{1-\frac{n}{2}}\left(1-\frac{a^2\xi^2}{k^2}
    \right)^{-1}\left[ I_{\frac{n}{2}-1}(aR)
        \frac{K_{\frac{n}{2}-2}\big(kR/\xi\big)}
        {K_{\frac{n}{2}}\big(kR/\xi\big)} +
        \frac{a\xi}{k}I_{\frac{n}{2}-2}(aR)
        \frac{K_{\frac{n}{2}-1}\big(kR/\xi\big)}{
            K_{\frac{n}{2}}\big(kR/\xi\big)}
            \right].\nonumber
\end{flalign}
\end{prop}

From now on, the subscripts of $B$ or $\tilde{B}$, that indicate
the sign of $A$, will be omitted when the discussion is relevant
to all three cases (similarly for other variables like the
coefficient matrix $M$ below).

Equipped with these expressions of the convolution, we further
study the existence/non-existence of the flock profile on the
parameter space. As mentioned above, the explicit formulas allow
us to write $\lambda_1$ and $\lambda_2$, by plugging
\eqref{eq:explicitconv} into \eqref{eq:Wrhosol}, in terms of
$\mu_1$, $\mu_2$, and $R$. Since $r^{1-n/2}I_{\frac{n}{2}-1}(kr)$ and
$r^{1-n/2}I_{\frac{n}{2}-1}(kr/\ell)$ are independent,
we deduce the formulas in Table \ref{table1}.

\begin{table}
\begin{center}
\begin{minipage}{0.8\textwidth}
\tabcolsep=8pt
\begin{tabular}{ccc}
\hline \hline
 & $\lambda_1$ & $\lambda_2$\\
\hline
$A>0$ &
$-C\frac{R^{\frac{n}{2}}}{k}\ell^{n-1}B_+(\ell)K_{\frac{n}{2}}(kR/\ell)$
& $\frac{R^{\frac{n}{2}}}{k}\ell^{n-1}B_+(1)K_{\frac{n}{2}}(kR)$
\\
$A=0$ &  $-C R^{\frac{n}{2}}
\ell^{n-1}B_0(\ell)K_{\frac{n}{2}}(kR/\ell)$ & $R^{\frac{n}{2}}
\ell^{n-1}B_0(1)K_{\frac{n}{2}}(kR)$ \\
$A<0$ &
$-C\frac{R^{\frac{n}{2}}}{k}\ell^{n-1}B_-(\ell)K_{\frac{n}{2}}(kR/\ell)$
& $\frac{R^{\frac{n}{2}}}{k}\ell^{n-1}B_-(1)K_{\frac{n}{2}}(kR)$\\
\hline
\end{tabular}
\end{minipage}
\end{center}
\caption{Formulas for $\lambda_1$ and $\lambda_2$ in~\eqref{eq:Wrhosol}
when $\rho$ is given by~\eqref{eq:explicitsol}.
}
\label{table1}
\end{table}

For the flock profile we are interested in, $\lambda_1$ and
$\lambda_2$ must be zero. In view of Table \ref{table1}, this is
equivalent to the conditions $B(\ell)=0, B(1)=0$, since $K_\nu(x)$ is nonzero on
$(0,\infty)$. Therefore, there exists a flock profile only if the
homogeneous equations for $\boldsymbol{\mu} = (\mu_1, \mu_2)^{T}$
\begin{equation}\label{eq:homomu}
 M \boldsymbol{\mu}
= \begin{pmatrix}
   \tilde{B}(\ell) &  1\cr
\tilde{B}(1) &  1
  \end{pmatrix}
\begin{pmatrix}
 \mu_1 \cr \mu_2
\end{pmatrix} =
\begin{pmatrix}
 0 \cr 0
\end{pmatrix}
\end{equation}
are satisfied. These two homogeneous equations, together with the
total unit mass constraint for the non-negative density $\rho$,
determine the three characterizing parameters $(\mu_1,\mu_2,R_F)$
of the flock profile.

A careful examination of the three equations shows that the radius
of the support $R$ is determined by the scalar equation $\det M =
\tilde{B}(\ell)-\tilde{B}(1)=0$, since $\mu_1$ and $\mu_2$ must be
nontrivial solutions of~\eqref{eq:homomu}. In fact, all the
subsequent results are based on studying the roots of $\det M$ and
the properties of $\tilde B(\xi)$ as functions of $R$.
Below we focus on the physical two- and three-dimensional cases, on
the biologically relevant regime $\ell<1,C\ell^{n-2}>1$. However,
unlike the unified derivation of the convolution
to~\eqref{eq:explicitconv}, the existence/non-existence question
is much more complicated and has to be treated separately.

\begin{figure}[ht]
\centering
\subfloat[Results of Section \ref{sec3},
$n=3$]{\includegraphics[keepaspectratio=true,
width=.5\textwidth]{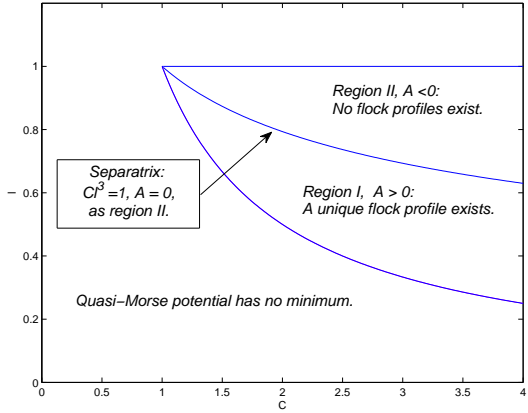}}
\subfloat[Results of Section \ref{sec2d}, $n=2$]
{\includegraphics[keepaspectratio=true,
width=.5\textwidth]{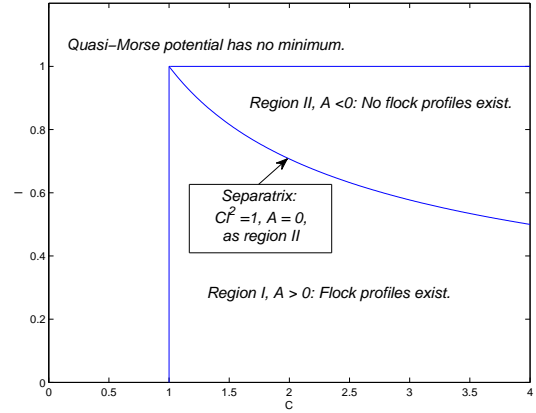}}
\caption{Phase-diagrams of parameters $C,\ell$ for the Quasi-Morse
potential illustrating the combined results of Theorems
\ref{theo3d} and \ref{theo2d}. For both dimensions $n=2,3$ the
aggregate parameter $A$ divides the biologically relevant parameter space $\{(C,\ell)\mid C\ell^{n-2}>1,\ell<1\}$ into
two subregions I and II by the curve $C\ell^n=1$. In region I, $A>0$, a flock profile always
exists. In region II and the separatrix, $A\leq 0$, no flock
profiles exist. When $n=3$, existing flock profiles are
additionally known to be unique.} \label{figresult}
\end{figure}

The main results of this paper (Theorems \ref{theo3d} and
\ref{theo2d}) in the biologically relevant regimes are summarized
in Figure \ref{figresult}. We show the existence and uniqueness of
flock profiles in the 3D case for $A>0$ and non-existence
otherwise. In the 2D case, we show the existence of flock
profiles for $A>0$ and non-existence otherwise. However, we
cannot conclude the uniqueness of the flock profiles. Because of
the connection of the (modified) Bessel functions in three
dimension (and odd dimensions in general) with the well-known
trigonometric functions, we consider this case first.


\section{Existence theory of flock profiles in three dimension}
\label{sec3}

We first turn to the existence theory of flock profiles in three
space dimensions, as in this case the Bessel functions in the
potential as well as in all subsequent computations reduce to
trigonometric functions (see Appendix~\ref{sec:bessel}). The
aggregate potential parameter $A$ is computed as
\begin{equation}
A = k^2\big(1-C\ell^3\big)/\big(C\ell^3-\ell^2\big),
\label{A3d}
\end{equation}
and the expressions \eqref{Bequationsgeneral} used in the explicit convolution \eqref{eq:explicitconv} simplify to
\begin{subequations}\label{eq:3dbb}
\begin{gather}
\tilde{B}_+(\xi) 
= \sqrt{\frac{2}{a\pi}}
\left(1+\frac{a^2\xi^2}{k^2}\right)^{-1}
\Big[\sin aR + \frac{a\xi}{k}\cos aR\Big]\frac{k}{kR+\xi},   \label{eq:3dbb+}\\
\tilde{B}_0(\xi)
= \frac{2\xi}{k^2}\frac{(kR)^2+3kR\xi+3\xi^2}{kR+\xi} + 1,  \label{B03dexpl}\\
\tilde{B}_-(\xi) 
= \sqrt{\frac{2}{a\pi}} \left(1-\frac{a^2\xi^2}{k^2}\right)^{-1}
\Big[\sinh aR + \frac{a\xi}{k}\cosh
aR\Big]\frac{k}{kR+\xi},\label{eq:3dbb-}
\end{gather}
\end{subequations}
as $K_{3/2}(x)/K_{1/2}(x) = 1+1/x$ and $K_{5/2}(x)/K_{3/2}(x) =
(x^2+3x+3)/x(x+1)$. Based on numerical findings, it has been
conjectured in~\cite{Carrillo2013} that flock profiles  can be found only for
Quasi-Morse potentials where $A>0$. The insight from the explicit calculations above enables us
now to prove existence and uniqueness of flock profiles, and thus
to analytically investigate the phase diagram of parameters
$C,\ell$ in the biologically relevant scenarios $C\ell>1, \ell<1$ (see
Figure \ref{figresult}). In fact, the following theorem holds:

\begin{theorem}
Let $W$ be a Quasi-Morse potential in space dimension $n=3$ with parameters within the biologically
relevant regime $C\ell>1,\ell<1$. Then flock profiles exist if and only if $A>0$.
Furthermore, if $A>0$, there exists a unique flock profile.
\label{theo3d}
\end{theorem}

To prove Theorem \ref{theo3d},
we begin with the discussion of the non-existence of flock profiles for $A\leq 0$.

\begin{proof}(Theorem \ref{theo3d}, Non-existence for $A\leq 0$)
When $A=0$, for all $R$, we can show $\det M_0 = \tilde{B}_0(\ell) - \tilde{B}_0(1) < 0$
by a straightforward explicit computation
using \eqref{B03dexpl}. We skip that calculation here as the case
$A=0$ will also be proven in general dimensions in Theorem
\ref{theo2d}.

Next, suppose that $A<0$. From \eqref{A3d}, this
implies $C\ell^3>\ell$ as $C\ell>1, \ell<1$ and furthermore, we have $a^2 = -A
=k^2(1-C\ell^3)/(\ell^2-C\ell^3)$.
%
The determinant of $M_-$ simplifies to
\begin{align*}
\det M_- =& \, \tilde{B}_-(\ell)-\tilde{B}_-(1)=
\sqrt{\frac{2}{\pi aR^2}}\frac{\ell^2(C\ell-1)}{1-\ell^2} \,\,\,\cdot \\
&\left[
\left(\frac{1}{C\ell^3}\frac{kR}{kR+\ell}-\frac{kR}{kR+1}\right)\sinh
aR
+\frac{a}{k}\left(\frac{1}{C\ell^2}\frac{kR}{kR+\ell}
-\frac{kR}{kR+1}\right) \cosh aR
\right] \cr
= &\,\sqrt{\frac{2}{\pi a}}\frac{k\ell^2(C\ell-1)}{1-\ell^2}
\frac{\cosh aR}{C\ell^3(kR+\ell)(kR+1)}f_-(R),
\end{align*}
where
\begin{equation}
f_-(R) = \frac{a\ell}{k}(1-C\ell^3)+kR(1-C\ell^3)\tanh aR +
(\ell-C\ell^3)aR + (1-C\ell^4)\tanh aR.
\label{fmR}
\end{equation}
Clearly, the sign of $\det M_-$ is determined by the sign of
$f_-(R)$. The first two terms in \eqref{fmR} are negative. If
$C\ell^4>1$, the last two terms are both negative as
$C\ell^3>1\Rightarrow C\ell^3>\ell$. If to the contrary $C\ell^4\leq 1$, the
sum of the last two terms in \eqref{fmR} satisfies
\[
(\ell-C\ell^3)aR + (1-C\ell^4)\tanh aR
< (1+\ell-C\ell^3-C\ell^4)aR
= (1+\ell)(1-C\ell^3)aR < 0,
\]
as $\tanh aR \leq aR$.
Thus $\det M_-<0$ for all $R>0$ and there is no real positive root
of $\det M_-$.
\end{proof}

\begin{figure}[htp]
 \begin{center}
\includegraphics[totalheight=0.25\textheight]{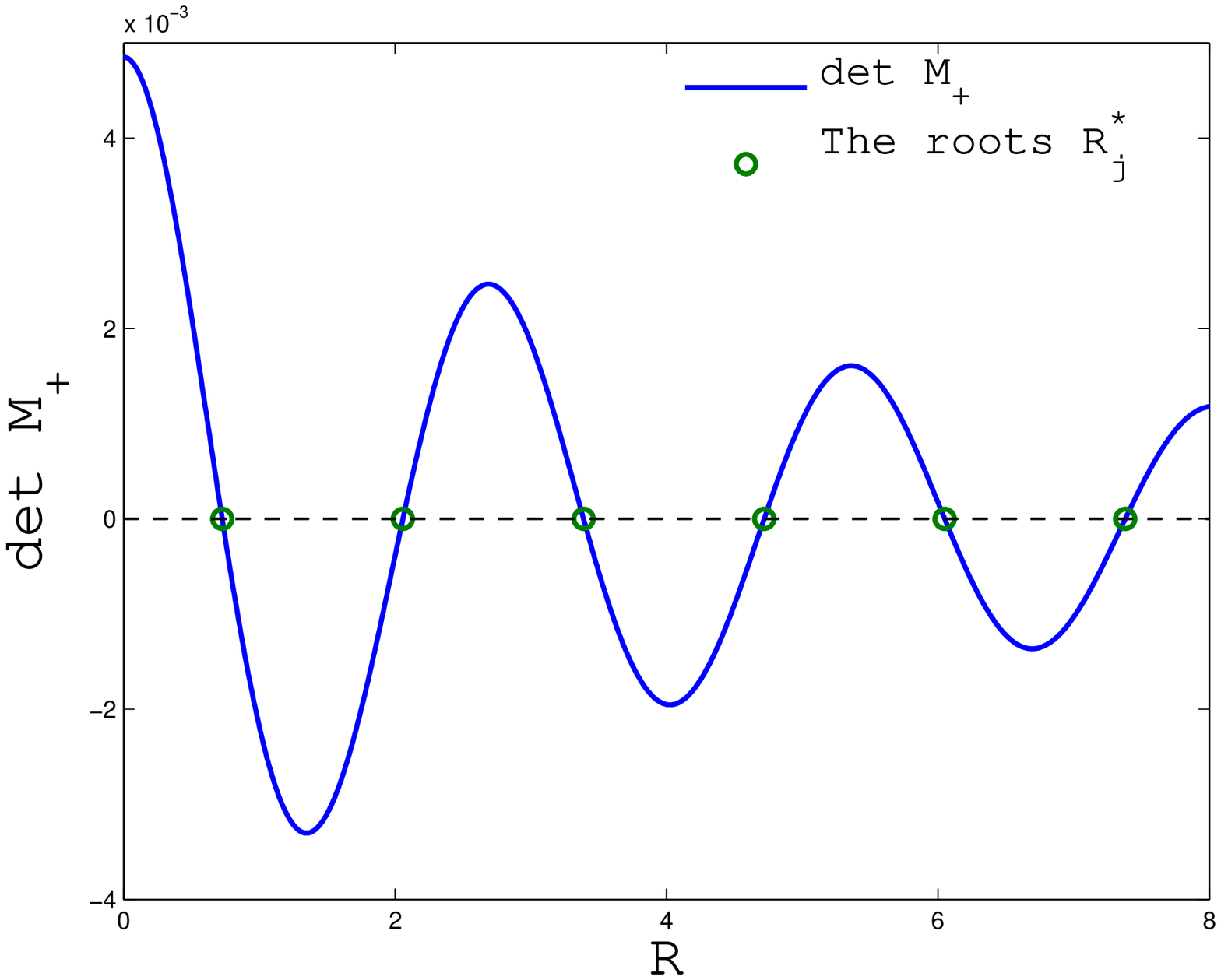}
$~$
\includegraphics[totalheight=0.25\textheight]{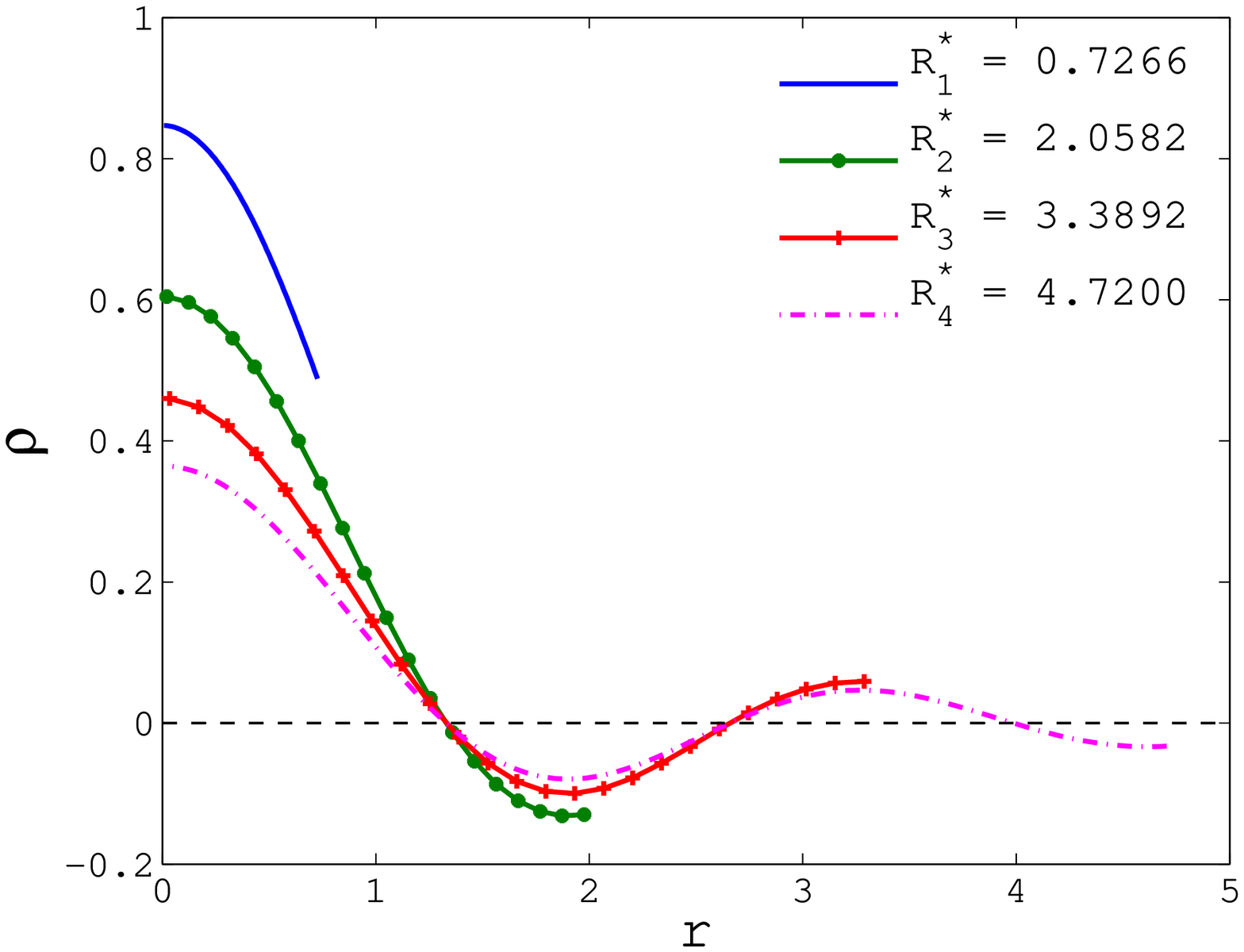}
 \end{center}
\caption{Multiple zeros $R^*$ of the equation
$\det M_+=0$ (left)
and the corresponding densities (right). Only the first zero
$R_1^*$ gives rise to strict positive density $\rho(r)$ on
the support. Here the parameters $C=1.255, \ell = 0.8,
k=0.2, A = 5.585$ (or $a=2.362$) are the same as in~\cite{Carrillo2013}.}
\label{fig:3ddetden}
\end{figure}

Proving existence of a unique flock profile when $A>0$ is more
difficult and relies on various properties of the  trigonometric
representation of the original half-integer order Bessel
functions. Our goal is to show that  $\det M_+$ is oscillatory
with decaying amplitude, implying the existence of infinitely many
positive roots $R^*_j$, $j=1,2,\dots$, for $\det M_+=0$. However,
only the first positive root gives rise to a strictly positive
density on the support $[0,R^*_1]$, and the density for any of the
other roots must be negative somewhere on the support $[0,R^*_j]$,
$j\ge 2$. This asserted behaviour of $\det M_+$ for $R>0$ is
illustrated in Figure~\ref{fig:3ddetden} with particular parameters
taken from~\cite{Carrillo2013}.

\begin{proof}(Theorem \ref{theo3d}, Existence and uniqueness for $A>0$.)
The proof is separated into several steps.

1. \emph{There are infinitely many positive roots for $\det
M_+=0$}. From \eqref{eq:3dbb+}, the determinant $\det M_+ =
\tilde{B}_+(\ell) - \tilde{B}_+(1)$ can be written as
\begin{align}
 \operatorname{det } M_+
=& \,k\sqrt{\frac{2}{a\pi}} \left(
\frac{1}{(1+a^2\ell^2/k^2)(kR+\ell)} -\frac{1}{(1+a^2/k^2)(kR+1)}
\right)\sin aR \nonumber \\ &+ \sqrt{\frac{2a}{\pi}}\left(
\frac{\ell}{(1+a^2\ell^2/k^2)(kR+\ell)}
-\frac{1}{(1+a^2/k^2)(kR+1)} \right)\cos aR . \label{eq:3dmp1}
\end{align}
We observe that the coefficient of $\sin aR$ in the expression
above is positive, since $ (1+a^2\ell^2/k^2)^{-1}>
(1+a^2/k^2)^{-1}$ and $(kR+\ell)^{-1}>(kR+1)^{-1}$. Evaluating
$\det M_+$ at $\tilde{R}_j=(j-1/2)\pi/a, j=1,2,\cdots$, the roots
of $\cos aR$, we deduce that
\begin{equation*}
\det M_+ \Big|_{R=\tilde{R}_j} = (-1)^{j}k\sqrt{\frac{2}{a\pi}}
\left( \frac{1}{(1+a^2\ell^2/k^2)(k\tilde{R}_j+\ell)}
-\frac{1}{(1+a^2/k^2)(k\tilde{R}_j+1)} \right)
\end{equation*}
has alternating signs. Therefore, there is at least one root
between $(\tilde{R}_j,\tilde{R}_{j+1})$, proving the existence of
infinitely many positive roots for $\det M_+=0$.

2. \emph{The function $\det M_+$ has no root on $(0,\tilde{R}_1)$
and has a unique root $R^*_j$ on $(\tilde{R}_j,\tilde{R}_{j+1})$,
$j=1,2,\cdots$}. We write $\det M_+$ in the following form,
\begin{align*}
 \frac{\mathrm{det } M_+}{\cos aR} &= k\sqrt{\frac{2}{a\pi}}
\left( \frac{1}{(1+a^2\ell^2/k^2)(kR+\ell)}
-\frac{1}{(1+a^2/k^2)(kR+1)} \right)\Big(\tan aR+ g(R)\Big) ,
\end{align*}
where
\begin{subequations}
\begin{align}
g(R) &= \frac{a}{k}
\frac{ (a^2\ell-k^2)kR+a^2\ell(\ell+1)}
{ a^2(\ell+1)kR+k^2+a^2(\ell^2+\ell+1)} \label{eq:gfun1}\\
&=  \frac{a}{k}
\left[ \frac{a^2\ell-k^2}{a^2(\ell+1)}
+ \frac{(k^2+a^2)(k^2+a^2\ell^2)}{
a^2(\ell+1)\big(a^2(\ell+1)kR+k^2+a^2(\ell^2+\ell+1)\big)
}
\right].
\end{align}
\end{subequations}
It is easy to check that the roots of $\det M_+ = 0$ are the same
as the roots of $\tan aR + g(R) = 0$, and this auxiliary function $g$
is used to show various estimates in various stages of the proof below.
Notice now that the function $\tan aR + g(R)$ is strictly increasing on
$(\tilde{R}_j,\tilde{R}_{j+1})$, since $\frac{d}{dR} \tan aR \geq
a$ and
\[
 g'(R) > g'(0) =
-a\frac{ (k^2+a^2)(k^2+a^2\ell^2)}{
\big(k^2+a^2(\ell^2+\ell+1) \big)^2}
>-a.
\]
Combining this with the fact that
\[
\lim_{R\to \tilde{R}_j^\mp}  \Big(\tan aR + g(R)\Big) =
\pm\infty,
\]
we obtain that there is a unique root $R^*_j$ on
$(\tilde{R}_j,\tilde{R}_{j+1})$, as illustrated in Figure~\ref{fig:3dill}(a). There is no positive root on
$(0,\tilde{R}_1)$, because $\det M_+$ is an increasing function on $(0,\tilde{R}_1)$
and
$$
\det M_+|_{R=0} = \sqrt{\frac{2a}{\pi}}\Big(
(1+a^2\ell^2/k^2)^{-1}-(1+a^2/k^2)^{-1}\Big)>0.
$$

\begin{figure}[htp]
    \begin{center}
\subfloat[The intersection of $\tan aR$ with $-g(R)$]{\includegraphics[keepaspectratio=true,
width=.5\textwidth]{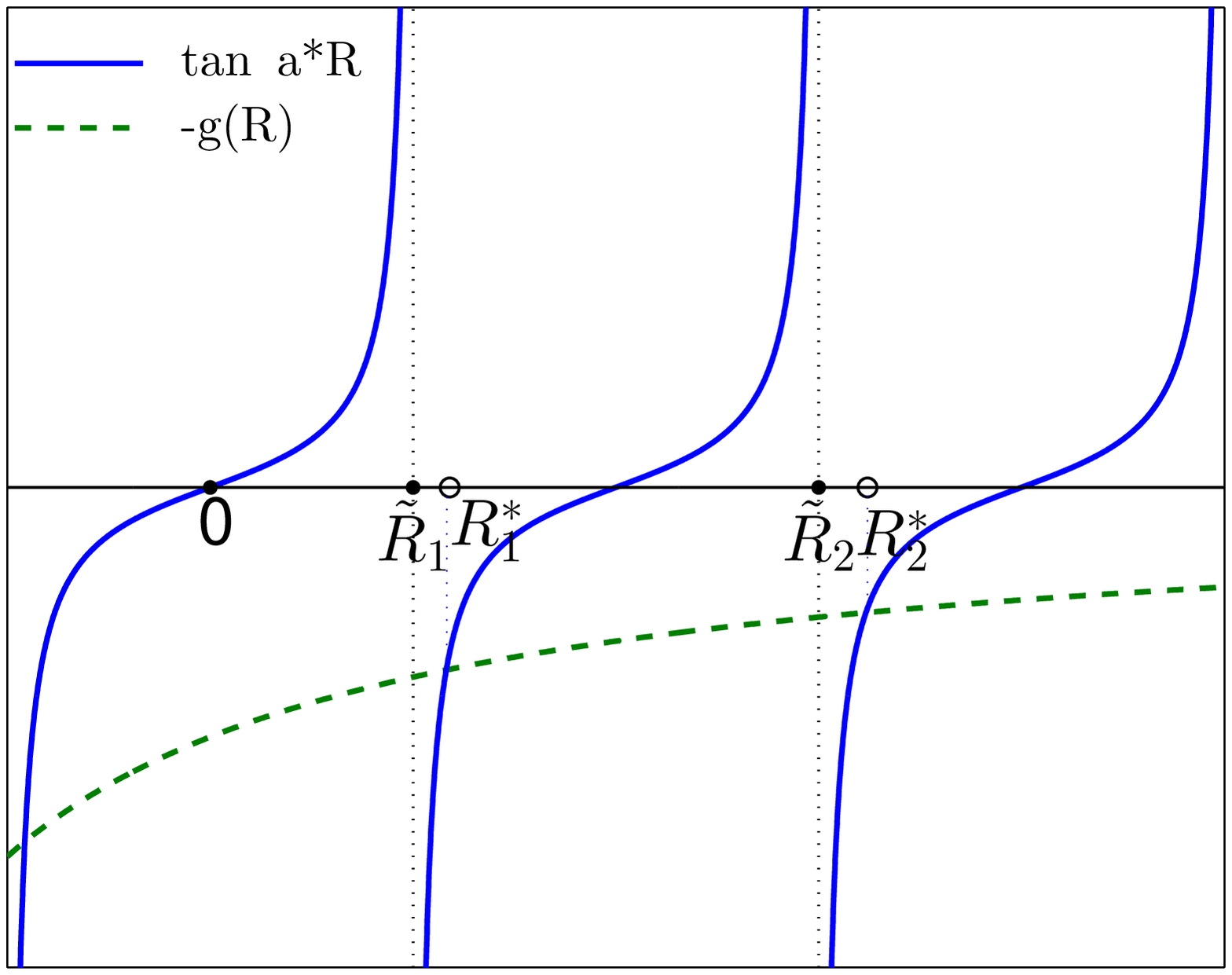}}
\subfloat[The densities corresponding to $R_1^*$ and $R_2^*$]
{\includegraphics[keepaspectratio=true,
width=.5\textwidth]{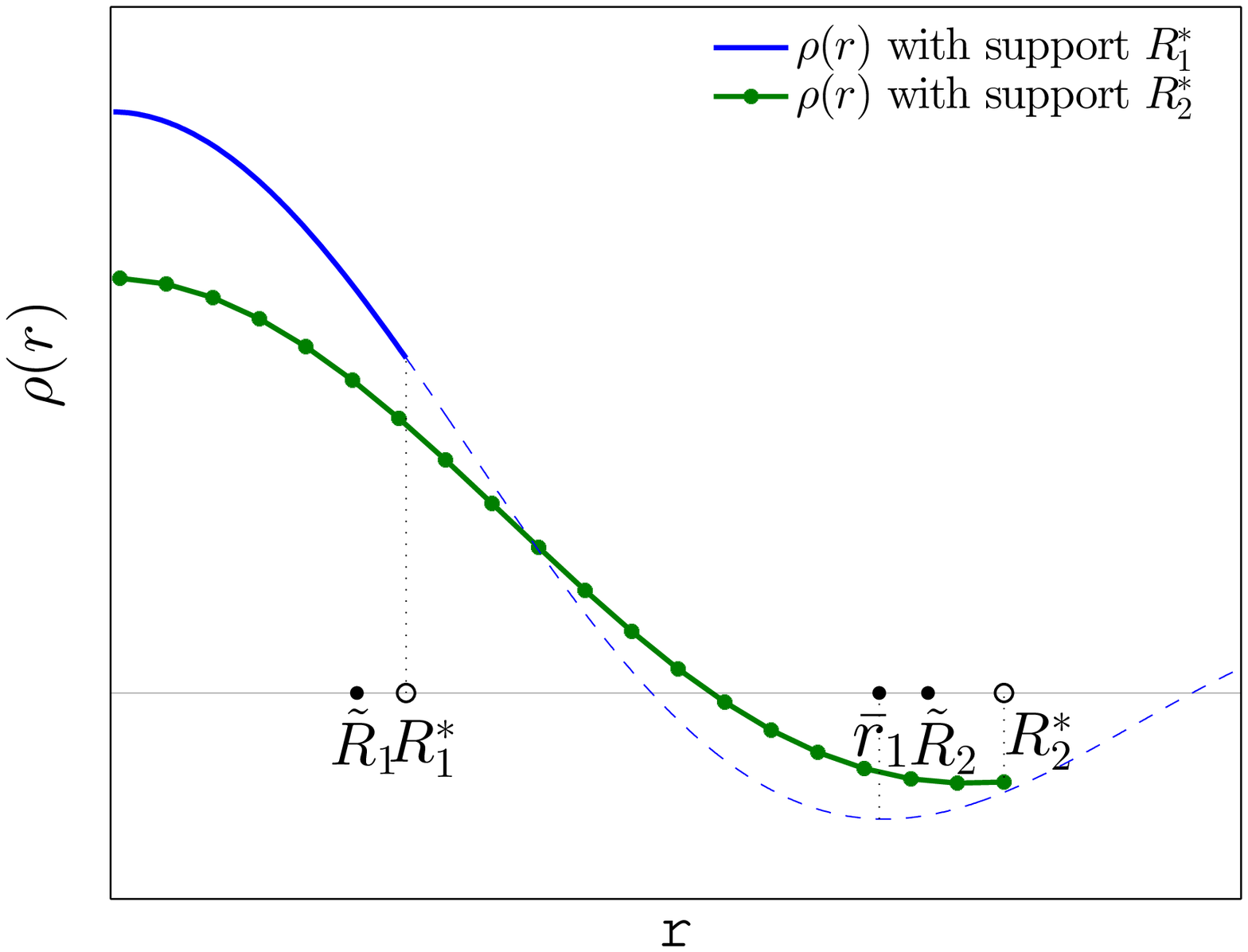}}
\end{center}
    \caption{Illustrations of the generic properties proved in the
        three dimensions when $A>0$: (a) $\tan aR$ and $g(R)$ intersects
        only once at $R_j^*$ in the interval $[\tilde{R}_j,\tilde{R}_{j+1})$;
        (b) The density $\rho(r)$ with support $R_j^*, j\geq 2$
        has opposite signs at the origin and at $\tilde{R}_2$ while 
        that with $R_1^*$ is monotonically
        decreasing from the origin.}
    \label{fig:3dill}
\end{figure}

3. \emph{If $j\ge 2$ then the density corresponding to the root
$R^*_j$ can not be both positive at the origin and at
$\tilde{R}_2$.} Let $\boldsymbol{\mu}=(\mu_1,\mu_2)^T$ be the
(nontrivial) solution of $M_+\big|_{R=R^*_j}\boldsymbol{\mu}=0$,
then the corresponding density is given by
$$
\rho(r) =\mu_1 r^{-1/2}J_{1/2}(ar) + \mu_2
=\mu_1\bigg(\sqrt{\frac{2}{a\pi}}\frac{\sin ar}{r}-\tilde{B}_+(1)\big|_{R=R^*_j}\bigg).
$$
A direct evaluation of $\rho$ leads to
\[
 \rho(0)\rho(\tilde{R}_2) =
 \left( \sqrt{\frac{2a}{\pi}}-\tilde{B}_+(1)\Big|_{R=R^*_j}\right)
 \left( -\sqrt{\frac{8a}{9\pi^3}}-\tilde{B}_+(1)\Big|_{R=R^*_j}\right)
\mu_1^2.
\]
Using~\eqref{eq:3dbb+} and the inequality $|\sin aR +
\frac{a\xi}{k}\cos aR| \leq (1+\frac{a^2\xi^2}{k^2})^{1/2}$, we
get
\begin{align*}
 |\tilde{B}_+(\xi)|
&= \left|\sqrt{\frac{2}{a\pi}} \left(1+\frac{a^2\xi^2}{k^2}\right)^{-1}
\Big[\sin aR +
\frac{a\xi}{k}\cos aR\Big]\frac{k}{kR+\xi} \right| \\
& \leq \sqrt{\frac{2}{a\pi}} \left(1+\frac{a^2\xi^2}{k^2}\right)^{-1/2}
\frac{k}{kR+\xi}.
\end{align*}
Therefore, since $R^*_j>\tilde{R}_2$,
\begin{align*}
 \Big|\tilde{B}_+(1)\big|_{R=R^*_j}\Big|
& \leq \sqrt{\frac{2}{a\pi}}\left(1+\frac{a^2}{k^2}\right)^{-1/2}\frac{k}{kR^*_j+1} <
\sqrt{\frac{2}{a\pi}}\frac{1}{\tilde{R}_2}
=\sqrt{\frac{8a}{9\pi^3}} < \sqrt{\frac{2a}{\pi}} .
\end{align*}
These estimates imply that $\rho(0)\rho(\tilde{R}_2)<0$, while the
physical density $\rho$ must be nonnegative on the support.

4. \emph{The density $\rho(r)$ corresponding to the root $R^*_1$
is decreasing and strictly positive on its support $[0,R^*_1]$}. Let us first
show that $\tilde{B}_+(\ell)|_{R=R^*_1}
=\tilde{B}_+(1)|_{R=R^*_1}<0$. Assume that this is not the case,
then $\tilde{B}_+(\ell)|_{R=R^*_1} =\tilde{B}_+(1)|_{R=R^*_1}\geq
0$. Since $\cos aR<0$ for $R\in (\tilde{R}_1,\tilde{R}_2)$, then
\[
\sin aR^*_1 + \frac{a\ell}{k}\cos aR^*_1 > \sin aR^*_1 +
\frac{a}{k}\cos aR^*_1 \geq 0.
\]
This, together with $(1+a^2\ell^2/k^2)^{-1}> (1+a^2/k^2)^{-1}$ and
$(kR^*_1+\ell)^{-1}>(kR^*_1+1)^{-1}$, implies that
$\tilde{B}_+(\ell)|_{R=R^*_1}>\tilde{B}_+(1)|_{R=R^*_1}\geq 0$,
leading to a contradiction. Therefore, combining
$\tilde{B}_+(1)|_{R=R^*_1}<0$ with the fact that
$\mu_2=-\tilde{B}_+(1)|_{R=R^*_1}\mu_1$ and
$\rho(0)=\sqrt{\frac{2a}{\pi}}\mu_1+\mu_2>0$,  both $\mu_1$ and
$\mu_2$ must be positive.

It is easy to check that $r^{-1/2}J_{1/2}(ar) =
\sqrt{\frac{2}{a\pi}}\frac{\sin ar}{r}$ is a decreasing function
till its first local minimum $\bar{r}_1$, determined by
\[
 0=\left.\frac{d}{dr} r^{-1/2}J_{1/2}(ar)
\right|_{r=\bar{r}_1} = \left.\sqrt{\frac{2}{a\pi}}
\frac{ar \cos ar -\sin ar}{r^2}
\right|_{r=\bar{r}_1} ,
\]
or equivalently $a\bar{r}_1 = \tan a\bar{r}_1>0$ with $\bar{r}_1
\approx 4.49/a \in (\tilde{R}_1,\tilde{R}_2)$. Using the definition~\eqref{eq:gfun1}
of $g$,
\[
\tan a\bar{r}_1 + g(\bar{r}_1) = a\bar{r}_1 +
g(\bar{r}_1) =
\frac{a^3}{k} \frac{(\ell+1)(1+k\bar{r}_1)(\ell+k\bar{r}_1)}{
a^2(\ell+1)k\bar{r}_1+k^2+a^2(\ell^2+\ell+1)}>0.
\]
Since $R^*_1$ is the
unique root of the strictly increasing function $\tan aR +g(R)$
on the interval $(\tilde{R}_1,\tilde{R}_2)$, the fact that $\tan
a\bar{r}_1 + g(\bar{r}_1) >0$ implies that $\bar{r}_1>R^*_1$.
Therefore, the density $\rho(r)$ is a decreasing function on
$[0,R^*_1]$, as illustrated in Figure~\ref{fig:3dill}(b). Finally, evaluating $\rho(r)$ at the boundary
$R=R^*_1$, we get
\begin{align*}
\rho(R^*_1) &= \mu_1\left( \sqrt{\frac{2}{a\pi}}\frac{\sin
aR^*}{R^*_1}-\tilde{B}_+(1)\Big|_{R=R^*_1} \right) \cr &=
-\mu_1\sqrt{\frac{2}{a\pi}}\left(1+\frac{a^2}{k^2}\right)^{-1}
\left[ \frac{a}{kR^*_1+1}+
\left(\frac{k}{kR^*_1+1}-\frac{1}{R^*_1}\left(1+\frac{a^2}{k^2}\right)
\right)\tan aR^*_1 \right]\cos aR^*_1 \cr &=
-\mu_1\sqrt{\frac{2}{a\pi}}\left(1+\frac{a^2}{k^2}\right)^{-1}
\left[ \frac{a}{kR^*_1+1}-
\left(\frac{k}{kR^*_1+1}-\frac{1}{R^*_1}\left(1+\frac{a^2}{k^2}\right)
\right)g(R^*_1) \right]\cos aR^*_1 \cr &=
-\mu_1\sqrt{\frac{2a}{\pi}} \frac{1}{kR^*_1}
\frac{1+\ell+kR^*_1}{a^2(\ell+1)kR^*_1+k^2+a^2(\ell^2+\ell+1)}
\cos aR^*_1
>0.
\end{align*}
This shows that $\rho(R^*_1)>0$, and therefore $\rho(r)$ is
strictly positive on its support, which completes the proof.
\end{proof}


\section{Existence theory of flock profiles in two dimension}
\label{sec2d}
We now turn our attention to two space dimensions, where the involved Bessel functions do not reduce to standard trigonometric expressions.
For $n=2$,
\begin{equation}
A = k^2(1-C\ell^2)/(C-1)\ell^2,
\label{A2d}
\end{equation}
 and
\begin{subequations}
\begin{gather}
    \tilde{B}_+(\xi) = \left(1+\frac{a^2\xi^2}{k^2}\right)^{-1}
    \left[
    J_0(aR)-\frac{a\xi}{k}J_1(aR)\frac{K_0(kR/\xi)}{K_1(kR/\xi)}
    \right], \label{eq:2db+}\\
    \tilde{B}_0(\xi) =
    \frac{2\xi}{k}R\frac{K_{5/2}(kR/\xi)}{K_{3/2}(kR/\xi)}+1,
    \label{eq:2db0}\\
    \tilde{B}_-(\xi) = \left(1-\frac{a^2\xi^2}{k^2}\right)^{-1}
    \left[
    I_0(aR)+\frac{a\xi}{k}I_1(aR)\frac{K_0(kR/\xi)}{K_1(kR/\xi)}
    \right]. \label{eq:2db-}
\end{gather}
\end{subequations}
The numerical investigations carried out in \cite{Carrillo2013} led to the assertion that flock profiles can only be found when $A>0$.
As in the three-dimensional case, we can now give a rigorous theorem and proof thanks to the explicit computations of Section \ref{sec2}.

\begin{theorem} \label{theo2d}
Let $W$ be a Quasi-Morse potential in space dimension $n=2$ with
parameters within the biologically relevant regime $C>1,\ell<1$.
Then flock profiles exist if and only if $A>0$ or equivalently
$C\ell^2 <1$.
\end{theorem}

We begin by proving a general monotonicity result on the ratio of
two modified Bessel functions, which will be used repeatedly
throughout the section.

\begin{lemma} For any $\nu\geq 0$, the functions
    $K_{\nu+1}(x)/\big(xK_\nu(x)\big)$,
    $K_\nu(x)/\big(xK_{\nu+1}(x)\big)$ and $K_{\nu+1}(x)/K_\nu(x)$
are strictly decreasing functions on $(0,\infty)$.
    \label{lem:monbesk}
\end{lemma}
\begin{proof}
    Let $w(x) = K_{\nu+1}(x)/(xK_{\nu}(x))$,
which is positive and
    smooth on $(0,\infty)$. We take the derivative of both
    sides of $K_{\nu+1}(x)=xw(x)K_{\nu}(x)$ and use 
    the recurrence relation
\[
    -K_{\nu}(x)-(\nu+1)K_{\nu+1}(x)/x=w(x)K_{\nu}(x)+xw'(x)K_{\nu}(x)+w(x)\big(
    \nu K_{\nu}(x)-xK_{\nu+1}(x)\big),
\]
which is equivalent to the differential equation for $w$
\begin{equation}\label{eq:weq}
2(\nu+1)w(x)+xw'(x)-x^2w(x)^2+1=0.
\end{equation}
Taking the derivative of~\eqref{eq:weq} w.r.t $x$,
\begin{equation}\label{eq:wdeveq}
(2\nu+3)w'(x)+xw''(x)-2xw(x)^2-2x^2w(x)w'(x)=0.
\end{equation}
We can first get the ``boundary conditions'' for $w$ near the
origin or infinity, by asymptotic expansions. When $x$ is close to
the origin, one uses \eqref{eq:bksmall} to deduce
\[
 w(x) \sim 2\nu x^{-2}, \quad
w'(x) \sim -4\nu x^{-3}<0,\quad
w''(x) \sim 12\nu x^{-4}>0,
\]
for $\nu>0$ and
\[
 w(x) \approx \frac{1}{x^2(-\frac{1}{2}\ln x- \gamma)}, \quad
 w'(x)  \sim \frac{4}{x^3\ln x}<0,\quad
 w''(x) \sim -\frac{12}{x^4\ln x} >0\,,
\quad
\]
for $\mu=0$. When $x$ is large, by the asymptotic
expansion~\eqref{eq:bklarge}, one gets
\[
    w(x) \sim \frac{1}{x}\left(1-\frac{2\nu+1}{2x}\right),\quad
    w'(x) \sim -\frac{1}{x^2}<0,\quad
    w''(x) \sim \frac{2}{x^3}>0.
\]
Therefore,  $w(x)>0, w'(x)<0, w''(x)>0$ when $x$ is near origin and $x \to
\infty$. Moreover, $w$ has no local maximum on $(0,\infty)$. Otherwise
if there is a local maximum at $x_0$, then $w'(x_0)=0$, $w''(x_0)\leq 0$.
On the other hand, by~\eqref{eq:wdeveq}, $w''(x_0) = 2w(x_0)^2 >0$,
a contradiction.

Next, we show that $w'(x)< 0$ on $(0,\infty)$. If
$w'(x)>0$  at some point $x_1>0$, then by the fact that
$w'(x)<0$ when $x$ is large, $w$ must have a local maximum on
$(x_1,\infty)$ (because $w$ first increases and then decreases).
If $w'(x)=0$ at $x_2>0$, then by~\eqref{eq:wdeveq},
$w''(x_2)=2w(x_2)^2>0$. Hence there is a point $\tilde{x}_2>x_2$,
such that $w'(\tilde{x}_2)>0$, and it is reduced to the previous case.
Therefore, in either situation, there exists a local maximum on
$(0,\infty)$, contradicting the statement proved in the last
paragraph. This concludes the proof of the strict monotonicity of
$w$ on $(0,\infty)$.

Similarly, the monotonicity of $w_2(x)=K_\nu(x)/(xK_{\nu+1}(x))$
and $w_3(x)=K_{\nu+1}(x)/K_{\nu}(x)$ can be proved, by using the
second-order ODEs
\[
    (2\nu-2)w_2'(x)+2x^2w_2(x)w_2'(x)+2xw_2(x)^2-xw_2''(x)=0
\]
and
\[
2x^2w_3(x)w_3'(x)+(2\nu+1)w_3(x)-2(\nu+1)xw_3'(x)-x^2w_3''(x)=0\,.
\]
In all the three cases, the key ingredients of the proof are the
right ``boundary condition'' near the origin and infinity, and
$w''(x)>0$ at any point $x$ such that $w'(x)=0$.
\end{proof}

Lemma \ref{lem:monbesk} is needed in the proof of Theorem
\ref{theo2d}, where contrary to the three-dimensional counterpart, the
ratios of Bessel functions do not simplify for even dimensions.
The structure of the proof given below would apply in a similar
fashion in three dimensions to obtain Theorem \ref{theo3d} if the
simplified expressions \eqref{eq:3dbb+}--\eqref{eq:3dbb-} were
omitted. We begin with a discussion of the case $A=0$ for any
dimensions.

\begin{proof}(Theorem~\ref{theo2d})
Suppose $A=0$. Then, in general dimension $n$,
\[
\det M_0 =\tilde{B}_0(\ell)-\tilde{B}_0(1) =
2R^2\left[ \frac{1}{kR/\ell} \frac{K_{\frac{n}{2}+1}(kR/\ell)}
{K_{\frac{n}{2}}(kR/\ell)} -
\frac{1}{kR} \frac{K_{\frac{n}{2}+1}(kR)}
{K_{\frac{n}{2}}(kR)}
\right]<0,
\]
as $\ell<1$ and the strict monotonicity of
$K_{\frac{n}{2}+1}(x)/(xK_{\frac{n}{2}}(x))$ is provided by
Lemma~\ref{lem:monbesk}. Hence, no real positive roots of $\det
M_0$ exist in any dimension. Let us return to the case $n=2$ and
suppose $A<0$, then $C\ell^2>1$ by \eqref{A2d} and $\det M$ can be
expressed as
\begin{equation}
\begin{gathered}
    \det M_- = \tilde{B}_- (\ell) - \tilde{B}_- (1) =
    \left[
\left(1-\frac{a^2\ell^2}{k^2}\right)^{-1}
-\left(1-\frac{a^2}{k^2}\right)^{-1}
    \right] I_0(aR) \cr
\quad +\frac{a}{k}\left[
\ell \left(1-\frac{a^2\ell^2}{k^2}\right)^{-1}
\frac{K_0(kR/\ell)}{K_1(kR/\ell)}
-\left(1-\frac{a^2}{k^2}\right)^{-1}
\frac{K_0(kR)}{K_1(kR)}
\right] I_1(aR) \cr
=
\frac{(C-1)(1-C\ell^2)}{C(1-\ell^2)} I_0(aR)
    + \frac{(C-1)a\ell^2}{k(1-\ell^2)}\left(
            \frac{1}{C\ell}\frac{K_0(kR/\ell)}{K_1(kR/\ell)}
            -\frac{K_0(kR)}{K_1(kR)}\right)
            I_1(aR),
\end{gathered}
\end{equation}
using~\eqref{eq:2db-}.
The coefficient of $I_0(aR)$ is obviously negative. By
the monotonicity of $K_0(x)/(xK_1(x))$,
\begin{align*}
 \frac{1}{C\ell}\frac{K_0(kR/\ell)}{K_1(kR/\ell)}
            -\frac{K_0(kR)}{K_1(kR)}
&<  \ell\frac{K_0(kR/\ell)}{K_1(kR/\ell)}
            -\frac{K_0(kR)}{K_1(kR)}\\
&< kR\left( \frac{\ell}{kR} \frac{K_0(kR/\ell)}{K_1(kR/\ell)}
            -\frac{1}{kR}\frac{K_0(kR)}{K_1(kR)}
\right)<0.
\end{align*}
This implies that $\det M_-<0$. Therefore, there is no
flock profile when $A\leq 0$.

Next, consider the case $A>0$. The determinant of the coefficient matrix is given as
\[
\det M_+ =
\frac{(C-1)(1-C\ell^2)}{C(1-\ell^2)}J_0(aR) -
    \frac{(C-1)a\ell^2}{k(1-\ell^2)}
    \left[
            \frac{1}{C\ell}\frac{K_0(kR/\ell)}{K_1(kR/\ell)}
            -\frac{K_0(kR)}{K_1(kR)}
    \right] J_1(aR).
\]
Let $0=\tilde{R}_0 < \tilde{R}_1 < \cdots$ be the simple zeros of
$J_1(aR)$, then by the relation $J_0'(x)=J_1(x)$, $\tilde{R}_j$
are also the critical points of $J_0(aR)$. Since $\det
M_+|_{R=\tilde{R}_j}$ has alternating signs, $\det M_+$ has at
least one root on $(\tilde{R}_j,\tilde{R}_{j+1})$ and therefore,
infinitely many roots on $(0,\infty)$.

Let ${R}^*$ be the first root in the first interval
$(\tilde{R}_0,\tilde{R}_1)$, then we must have
$\tilde{B}_+(\ell)|_{R=R^*} = \tilde{B}_+(1)|_{R=R^*}<0$
as illustrated in Figure~\ref{fig:extflock2d}(a).
Otherwise, if $\tilde{B}_+(\ell)|_{R=R^*} =
\tilde{B}_+(1)|_{R=R^*}\geq 0$, using \eqref{eq:2db+} we deduce
\[
J_0(aR^*) \geq \frac{a\ell}{k} J_1(aR^*)
\frac{K_0(kR^*/\ell)}{K_1(kR^*/\ell)},\quad J_0(aR^*) \geq
\frac{a}{k} J_1(aR^*) \frac{K_0(kR^*)}{K_1(kR^*)}.
\]
On the other hand, since $J_1(aR^*)$ is positive together with the
monotonicity of $K_0(x)/(xK_1(x))$,
\[
J_0(aR^*) - \frac{a\ell}{k}J_1(aR^*) \frac{K_0(kR^*/\ell)}{K_1(kR^*/\ell)} >
J_0(aR^*) - \frac{a}{k}J_1(aR^*) \frac{K_0(kR^*)}{K_1(kR^*)} \geq 0,
\]
and consequently,
\begin{align}
\tilde{B}_+(\ell) &= \left(1+\frac{a^2\ell^2}{k^2}\right)^{-1}\left[
J_0(aR^*) - \frac{a\ell}{k}J_1(aR^*) \frac{K_0(kR^*/\ell)}{K_1(kR^*/\ell)}\right] \cr
&> \left(1+\frac{a^2}{k^2}\right)^{-1}\left[
J_0(aR^*) - \frac{a}{k}J_1(aR^*) \frac{K_0(kR^*)}{K_1(kR^*)}\right]
=\tilde{B}_+(1),
\end{align}
contradicting the fact that $R^*$ satisfies $\det M_+|_{R=R^*}=
\tilde{B}_+(\ell)|_{R=R^*} - \tilde{B}_+(\ell)|_{R=R^*}=0$.

Since $\mu_2=-\tilde{B}_+(1)|_{R=R^*}\mu_1$, then $\mu_1$ and
$\mu_2$ have the same sign. If the corresponding density $\rho(r) =
\mu_1 J_0(ar)+\mu_2$ at the origin is nonnegative, then both
$\mu_1$ and $\mu_2$ are positive. We first factor out $J_0(aR^*)$
from the equation $\tilde{B}_+(\ell)|_{R=R^*} -
\tilde{B}_+(\ell)|_{R=R^*}=0$, i.e.,
\[
J_0(aR^*) = \frac{1}{ak(1-\ell^2)}\left[
\ell(k^2+a^2)\frac{K_0(kR^*/\ell)}{K_1(kR^*/\ell)}
-(k^2+a^2\ell^2)\frac{K_0(kR^*)}{K_0(kR^*)}
\right].
\]
Substituting this into $\rho(R^*) = \mu_1 J_0(aR^*) +\mu_2 =
\mu_1(J_0(aR^*)-\tilde{B}_+(\ell)|_{R=R^*})$, we conclude
\[
\rho(R^*) = \frac{a\ell^2}{k^2R(1-\ell^2)}\left[
\frac{kR^*}{\ell}\frac{K_0(kR^*/\ell)}{K_1(kR^*/\ell)}
-kR^*\frac{K_0(kR^*)}{K_1(kR^*)}
\right] J_1(aR^*)\mu_1 > 0.
\]
Finally, since $R^*$ is smaller than the first local minimum
$\tilde{R}_1$ of $J_0(ar)$, $\rho(r) = \mu_1 J_0(aR)+\mu_2$ is
decreasing on $[0,R^*]$. Thus, the strict positivity of $\rho(r)$
on $[0,R^*]$ results from the strict positivity of $\rho(R^*)$.
\end{proof}

\begin{figure}[htp]
 \begin{center}
 \subfloat[The intersection of $\tilde{B}_+(\ell)$
 and $\tilde{B}_+(1)$ at $R_j^*$]{\includegraphics[keepaspectratio=true,
width=.5\textwidth]{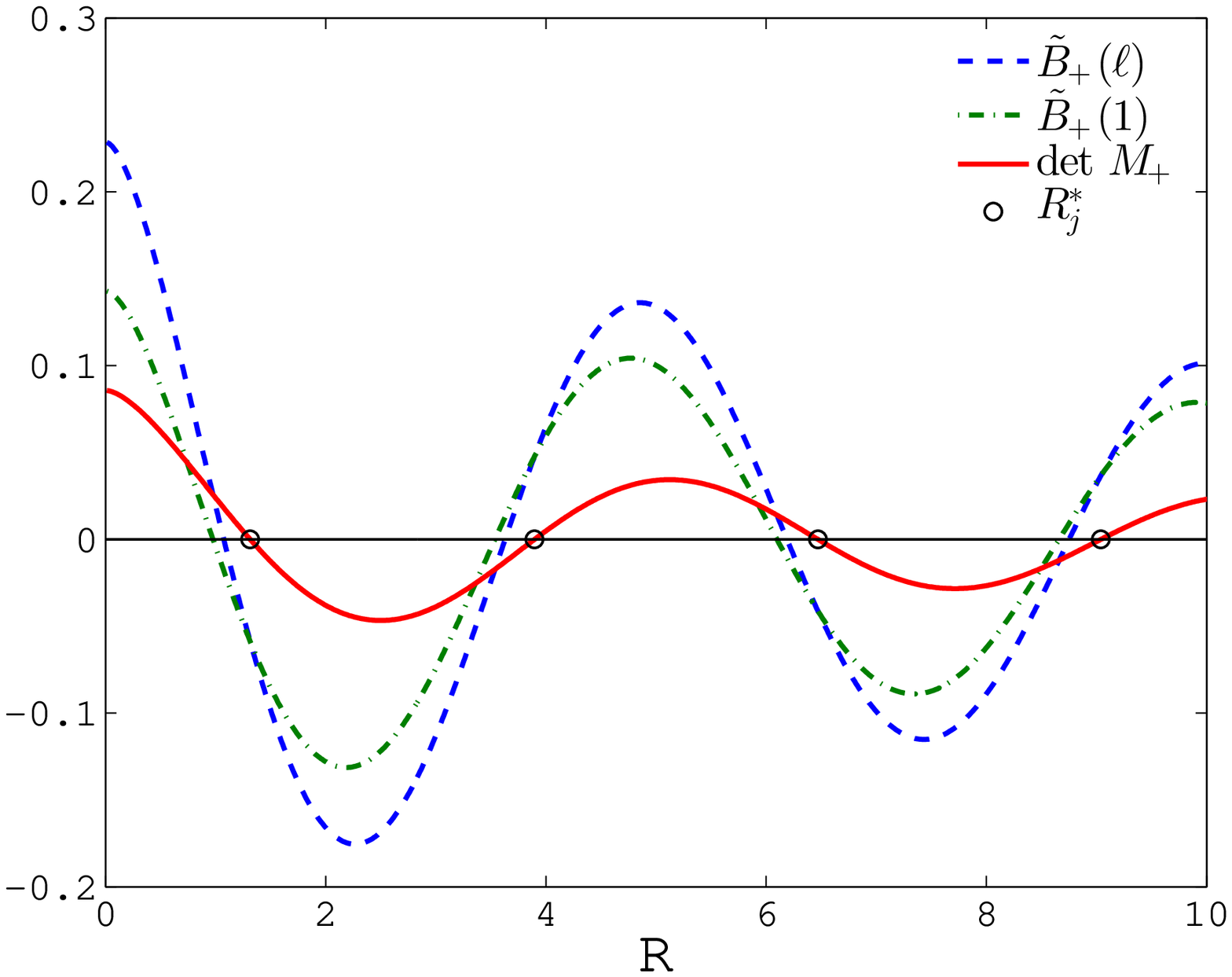}}
\subfloat[The densities corresponding to $R_j^*$]
{\includegraphics[keepaspectratio=true,
width=.5\textwidth]{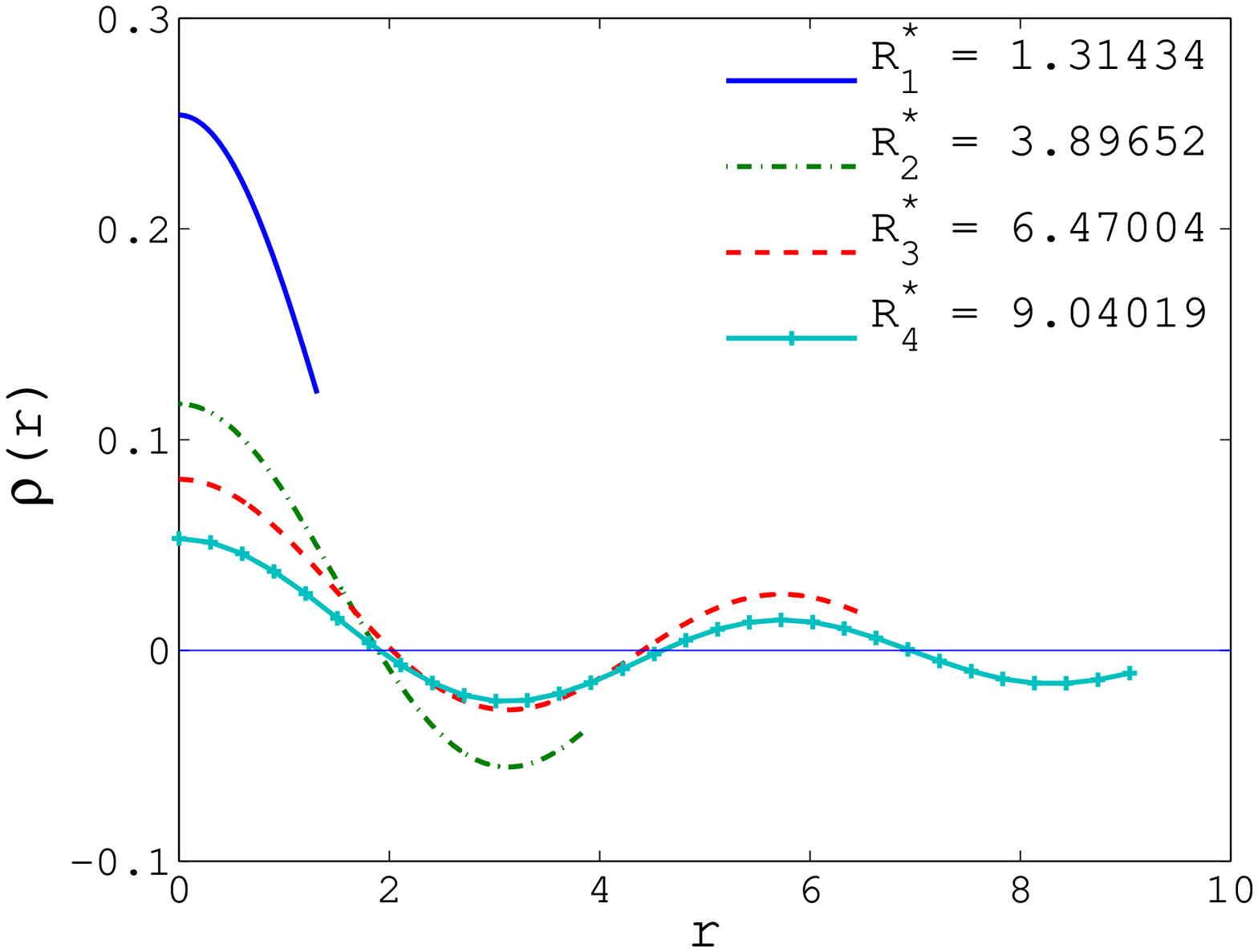}}
 \end{center}
\caption{The roots of the determinant $M_+$ and the corresponding
flock profiles. Only the first zero $R_1^*$ is physically relevant,
as the densities become negative on the support $(0,R_k^*)$ for the
other roots $R_k^*$. The parameters $C=10/9$, $\ell=0.75$, $k=1/2$ and $A=1.5$
are the same as in~\cite{Carrillo2013}.
}
\label{fig:extflock2d}
\end{figure}

\begin{rem}
Theorem \ref{theo2d} lacks the uniqueness result of Theorem
\ref{theo3d}. However, numerical investigations point towards a
uniqueness result similar to three dimensions. As an example, we
illustrate $\det M_+$ and the densities associated to its roots
for a set of parameters investigated in \cite{Carrillo2013} in
Figure \ref{fig:extflock2d}. To prove uniqueness in two
dimensions, the possibility of nonnegative densities for roots
$R^*>\tilde{R}_2$ and the possibility of multiple solutions $\det
M_+=0$ in $(\tilde{R}_0,\tilde{R}_1)$ have to be ruled out.
\end{rem}


\section{Further properties of flock profiles for the Quasi-Morse potential}

Let us remark that there are parameters $(C,\ell)$ such that the
convolution equation \eqref{eq:steadyeq} has a solution even
though they do not belong to the biologically relevant cases.
Flock profiles, as defined in Definition 1.1, can be found by
similar proofs as in the previous two sections in the region
$\{(C,\ell) \mid \ell>1, C\ell^{n-2}>1,  C\ell^{n}<1\}$, where $U$
has a positive global maximum. This family of flock profiles are
in fact those that are corresponding stable steady solution in the
time-reversed first-order swarming system \eqref{eq:firstorder},
and are not observed in simulations, since they are unstable, both
for first-order and second-order particle models.

The proofs in the previous two sections also indicate the
dependence of the flock profiles with respect to the size of their
support $R^*$ parameterized by $\ell$, at least in the asymptotic
limit of $\ell$ approaching its lower and upper limit. For example
in 3D, since $R^* \in (\tilde{R}_1,\tilde{R}_2)$
 and $\tilde{R}_j \sim O(a^{-1})$, we have $R^* \sim O(a^{-1})$.

In three dimensions, for fixed parameters $C$ and $k$, if $\ell$ is
close to its upper limit $C^{-1/3}$ in the parameter space, then
$a = k\sqrt{(1-C\ell^3)/(C\ell^3-\ell^2)}$ is close to zero, and
for the auxiliary function $g(R)$ defined in~\eqref{eq:gfun1}, we
have
\[
g(R) = \frac{a}{k} \frac{ (a^2\ell-k^2)kR+a^2\ell(\ell+1)} {
a^2(\ell+1)kR+k^2+a^2(\ell^2+\ell+1)} \approx -aR.
\]
The desired root $R^*$ can be approximated from the simplified
equation $\tan a R -aR = 0$, which is simply
$R^*\approx\bar{r}_1\approx 4.49/a$ in the last step of the proof
of Theorem~\ref{theo3d}. Therefore, as $\ell$ increases to
$C^{-1/3}$, the radius of support of the flock profile also
approaches the first minimum of $r^{-1/2} J_0(ar)$.

On the other hand, if $\ell$ is close to its lower limit $C^{-1}$,
$a$ diverges, and
\[
g(R) \approx \frac{a}{k}\frac{\ell k
R+\ell^2+\ell}{(\ell+1)kR+\ell^2+\ell+1}.
\]
Since $\ell$ is close  to $C^{-1}$ and the desired root $R^* \sim
a^{-1}$ is close to zero, $g(R)$ can be further simplified to
\[
g(R) \approx \frac{a}{k}\frac{C+1}{C^2+C+1} := a\bar{C},
\]
a constant proportional to $a$. From the asymptotic equation $\tan
aR^*+ a \bar{C}=0$, $aR^*$ approaches $\pi/2$ from above, or $R^*
\approx \pi/(2a)$.

Summarizing, in term of the original parameters $k$, $C$  and
$\ell$,
\begin{equation}\label{eq:Rupper}
R^* = \frac{4.49\sqrt{1-C^{-2/3}}}{k}(1-C\ell^3)^{-1/2} +
O(|1-C\ell^3|)
\end{equation}
when $\ell$ is close to $C^{-1/3}$ and
\begin{equation}\label{eq:Rlower}
R^* =\frac{\pi}{2k\sqrt{C^2-1}}(C\ell-1)^{1/2}+O(|C\ell-1|)
\end{equation}
when $\ell$ is close to $C^{-1}$. The comparison between these
asymptotic expansions of $R^*$ with those obtained from solving
$\det M_+=0$ by a root-finding algorithm is shown in Figure~\ref{fig:3dasymp}.
Substituting the above expressions into $M_+$, the expansions
for $\mu_1$ and $\mu_2$ can be obtained accordingly.

\begin{figure}[htp]
\includegraphics[totalheight=0.26\textheight]{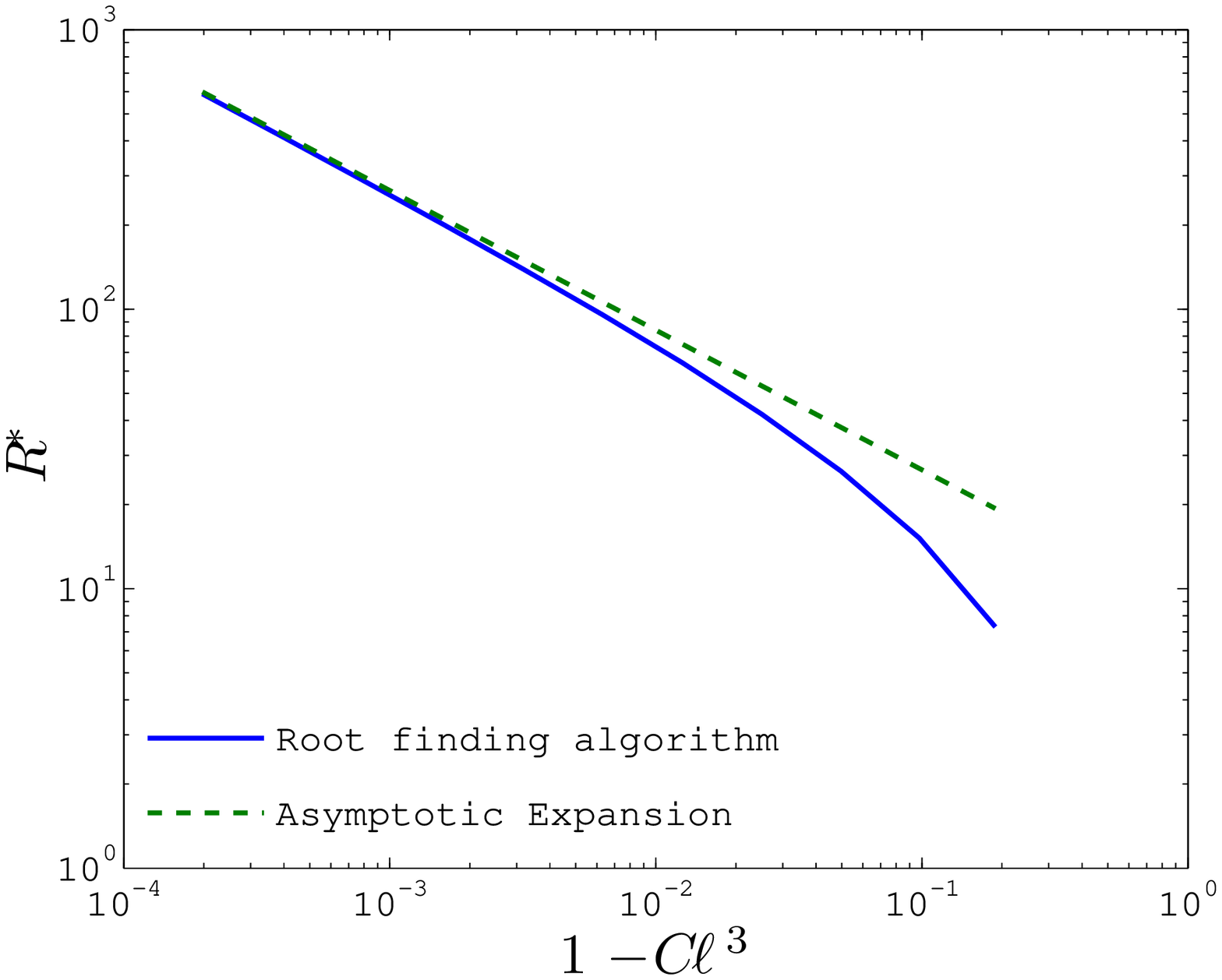}
\includegraphics[totalheight=0.26\textheight]{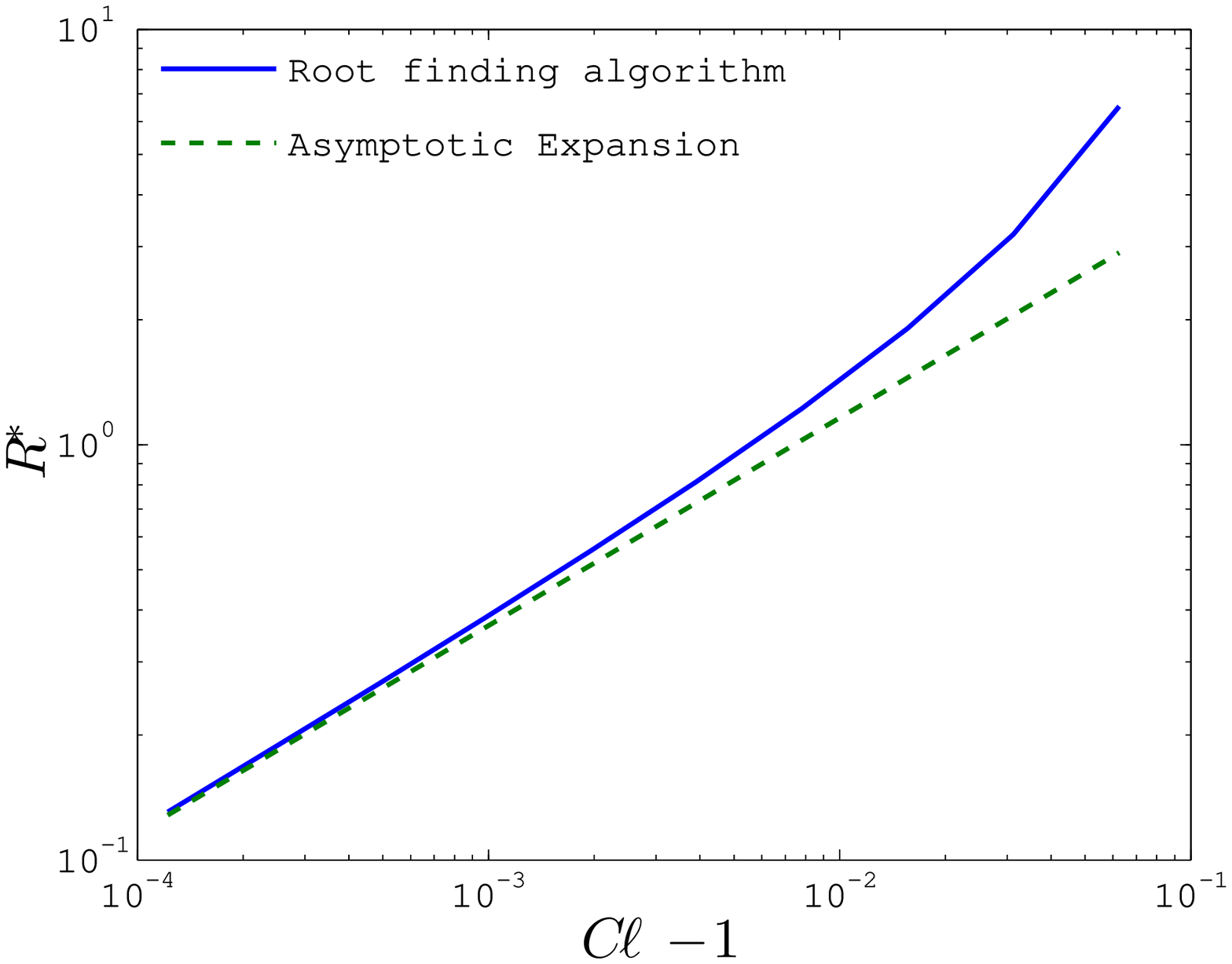}
\caption{The comparison between the radius of support $R^*$ by a
root finding algorithm of $\det M_+=0$ and the asymptotic
expansion given by~\eqref{eq:Rupper} and ~\eqref{eq:Rlower}. }
\label{fig:3dasymp}
\end{figure}

In two dimensions, the leading-order asymptotic expansion of $R^*$
can be derived similarly. When $\ell$ is close to zero, $a \approx
k/(\ell\sqrt{C-1})$ is large and $R^*\sim a$. Assuming $R^* = \ell
R_0 + O(\ell^2)$ for some $R_0>0$, then
\[
\tilde{B}_+(\ell)|_{R=R^*} \approx \frac{C-1}{C}\left[
J_0\big(kR_0/\sqrt{C-1}\big)-\frac{1}{\sqrt{C-1}}
J_1\big(kR_0/\sqrt{C-1}\big)\frac{K_0(kR_0)}{K_1(kR_0)}
\right]=O(1)
\]
and
\[
\tilde{B}_+(1)|_{R=R^*} \approx \ell^2(C-1)\left[
J_0\big(kR_0/\sqrt{C-1}\big) -
\frac{1}{\ell\sqrt{C-1}}J_1\big(kR_0/\sqrt{C-1}\big)
\frac{K_0(kR_0/\ell)}{K_1(kR_0/\ell)} \right].
\]
Since $\frac{K_0(kR_0/\ell)}{K_1(kR_0/\ell)} \to 1$ as $\ell \to
0$, we have
 $\tilde{B}_+(1)|_{R=R^*} = O(\ell)$ and $\tilde{B}_+(\ell)|_{R=R^*}
 \gg \tilde{B}_+(1)|_{R=R^*}$ unless the leading order in
 $\tilde{B}_+(\ell)|_{R=R^*}$ vanishes. Therefore, the coefficient $R_0$ is determined by
 \[
 J_0\big(kR_0/\sqrt{C-1}\big)=\frac{1}{\sqrt{C-1}}
J_1\big(kR_0/\sqrt{C-1}\big)\frac{K_0(kR_0)}{K_1(kR_0)},
 \]
where the positive number $kR_0/\sqrt{C-1}$ is smaller than the
first positive root of $J_0$ since this equation has infinitely
many roots.

When $\ell$ is close to $C^{-1/2}$, $a$ is small and $\det M_+$ is
\[
 \frac{a^2(1-\ell^2)}{k^2(1+a^2/k^2)(1+a^2\ell^2/k^2)}J_0(aR^*)-
\frac{a(C-1)\ell^2}{k(1-\ell^2)}\left[
\frac{1}{C\ell}\frac{K_0(kR^*/\ell)}{K_1(kR^*/\ell)}-
\frac{K_0(kR^*)}{K_1(kR^*)} \right]J_1(aR^*).
\]
From the fact that $R^*$ diverges,
\[
 \frac{1}{C\ell}\frac{K_0(kR^*/\ell)}{K_1(kR^*/\ell)}-
\frac{K_0(kR^*)}{K_1(kR^*)} \to \frac{1}{C\ell}-1 \approx C^{-1/2}
-1 \neq 0.
\]
Therefore, $\det M_+=0$ only if $J_1(aR^*)$ vanishes to have both
terms above of order $a^2$. In other words, $R^*$ converges to the
first positive root of $J_1(ar)$. Consequently, the expansions of
$R^*$ in two dimensions can be obtained.


\section{Variants of Morse-type potentials}

In the previous sections, we have shown that flock profiles
precisely exist for the Quasi-Morse potential when the parameters
$C$ and $\ell$ are in the region $\{ (C,\ell)\mid C\ell^{n-2}>1,
\ell<1, C\ell^{n}<1\}$, see Figure~\ref{figresult}. The conditions
$C\ell^{n-2}>1$ and $\ell<1$ ensure that the potential $U(r)$ is
\emph{biologically relevant} since it has a positive global
minimum, while the condition $C\ell^{n}<1$ is related to the 
\emph{non-H-stability} of the potential. A similar result for the
Morse-potential is presented in~\cite{d2006self}. The claim,
that a positive global minimum of the potential and non-H-stability imply
existence of compactly supported flock solutions, also seems to be
true for other similar potentials of the form
$U(r)=V(r)-CV(r/\ell)$, but concentration of density may appear
and the dimensionality of the support can vary with $U$. We show
some numerical evidence in support of the claim for the
generalised Morse-like potential with
\begin{equation}
V(r)=-e^{-\frac{r^p}{p}},\qquad p>0.\label{morselike}\end{equation}

For this potential, the non-H-stability condition $C\ell^{n}<1$ is
the same but the biologically relevant region is given by $\ell<1$
and $C>\ell^p$. The numerical simulations were conducted by
finding stationary profiles of the first-order swarming system of
particles given by
\begin{equation}\label{eq:firstorder}
    \frac{dx_i}{dt}=-\frac1N \sum_{j\neq i} \nabla W\left(x_i-x_j\right),
\quad i=1,\dots, N.
\end{equation}
Taking these positions and the common velocity $u_0$
with $|u_0|^2=\alpha/\beta$ as initial data
for the second-order system \eqref{eq:partsys2}, the resulting
stationary solution is stable~\cite{CHM2}.

In Figure~\ref{fig:general} (a), we observe generic
non-concentrated compactly supported flock profiles for the
exponent $p=\frac{1}{2}$ and $\ell <\ell^*= C^{1/p}=0.36$ that appear to converge to a continuous
distribution as $N\rightarrow \infty$. The same phenomena are
observed for exponents $p \in (0,1)$.

However, this type of aggregation cannot be expected for exponents
$p \in (1,2)$. For $C<1$, the density seems to concentrate towards
its boundary when $\ell$ approaches $\ell^*=C^{1/p}$, as illustrated in
Figure~\ref{fig:general} (b). For $C>1$, we observe mixed
dimensionality of the support in Figure~\ref{fig:general} (c) for
varying exponents $p$ approaching the limit case $p=2$. Flock
profiles seem to bifurcate as $p\to 2$ leading to a concentration
on a ring plus a continuous distribution inside. To our knowledge
this surprising phenomenon of mixed dimensionality of the support
has only been reported in 3D simulations in
\cite{soccerball,BCLR2} for purely attractive-repulsive potentials. In a swarming model of locusts in 2D using Morse 
potential~\cite{Bernofflocust,MR2788924,MR3124884}, the concentration of densities on the (one-dimensional) ground can also be reproduced
from obervations in nature, by including additional external gravity force. 

\begin{figure}[htp]
    \begin{center}
    \subfloat[$p=1/2, C = 0.6$]{\includegraphics[keepaspectratio=true,
width=.33\textwidth]{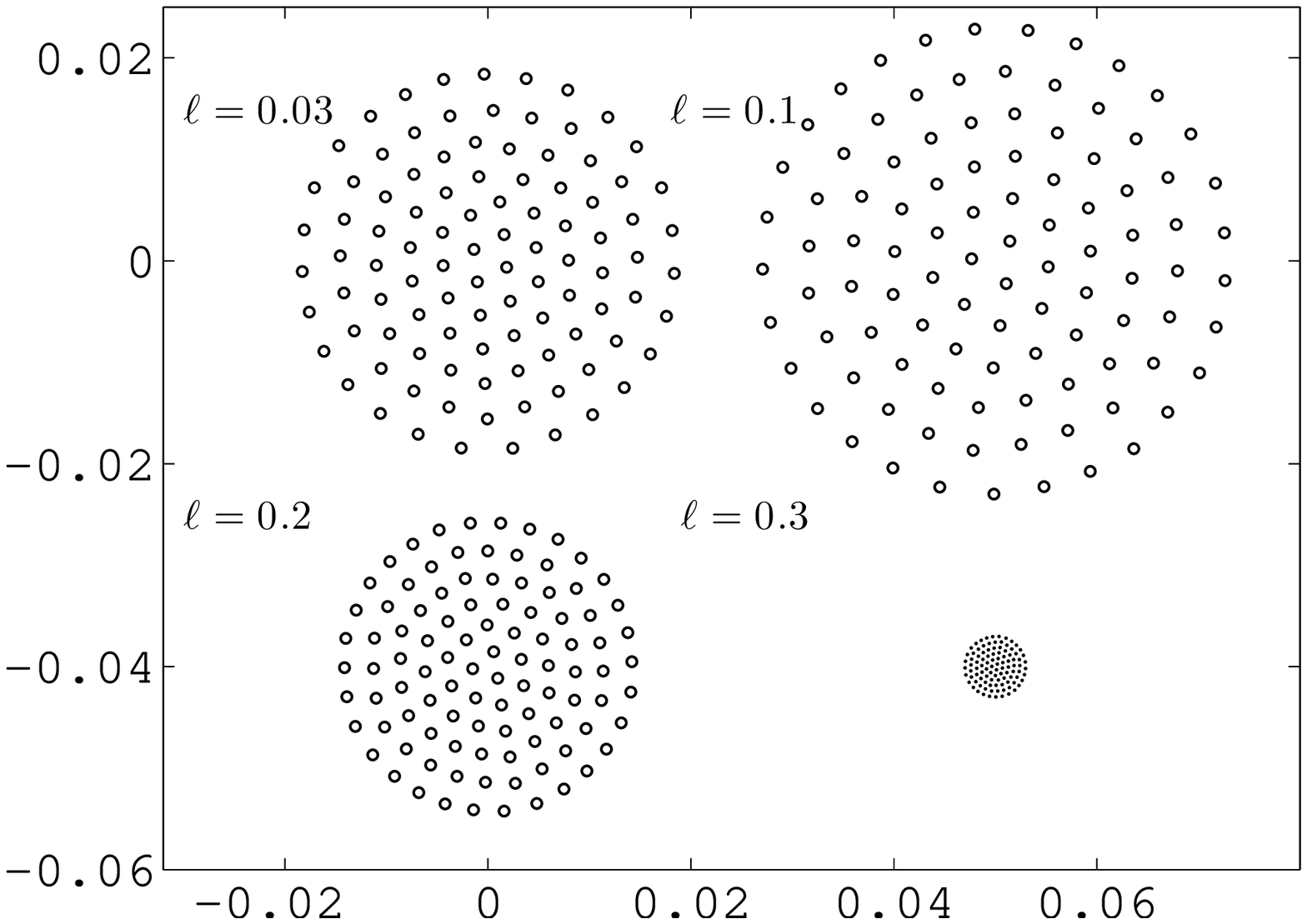}}
$~~$
\subfloat[$p=3/2, C = 0.6, \ell^*=C^{1/p}=0.7114$]
{\includegraphics[keepaspectratio=true,
width=.62\textwidth]{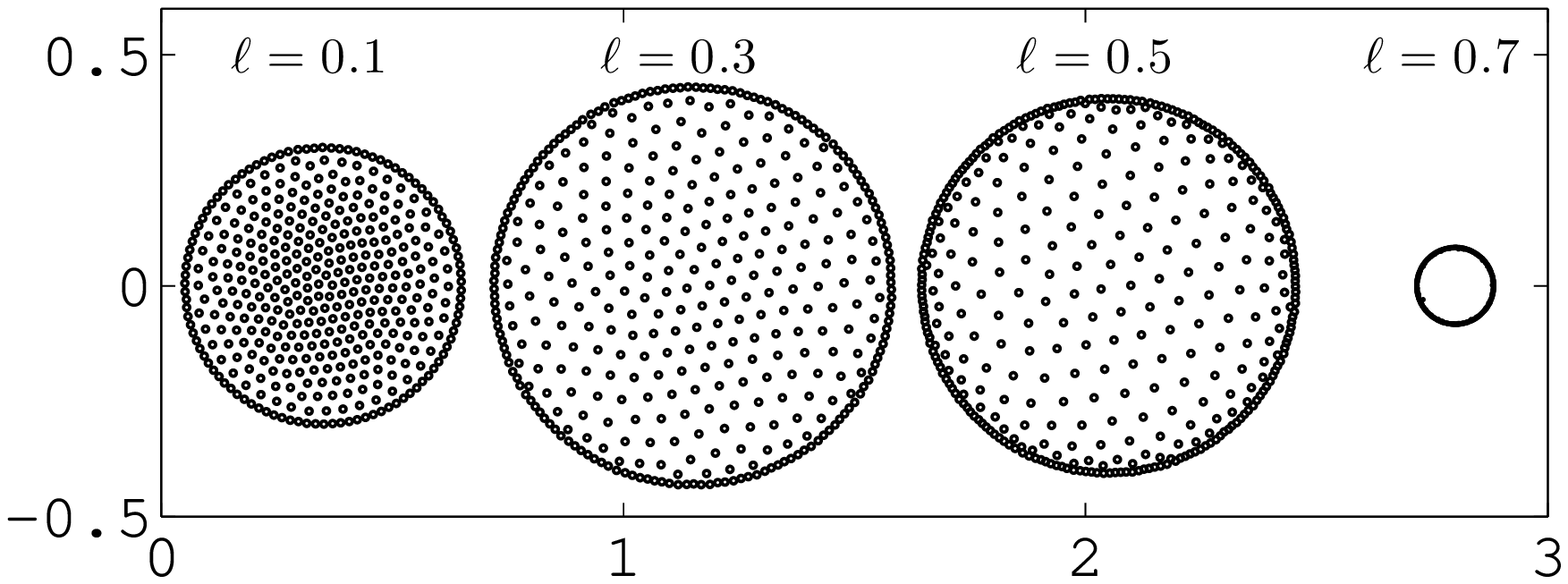}} \\
\subfloat[Different $p$'s with $C=10/9, \ell=3/4$]
{\includegraphics[keepaspectratio=true,
width=.8\textwidth]{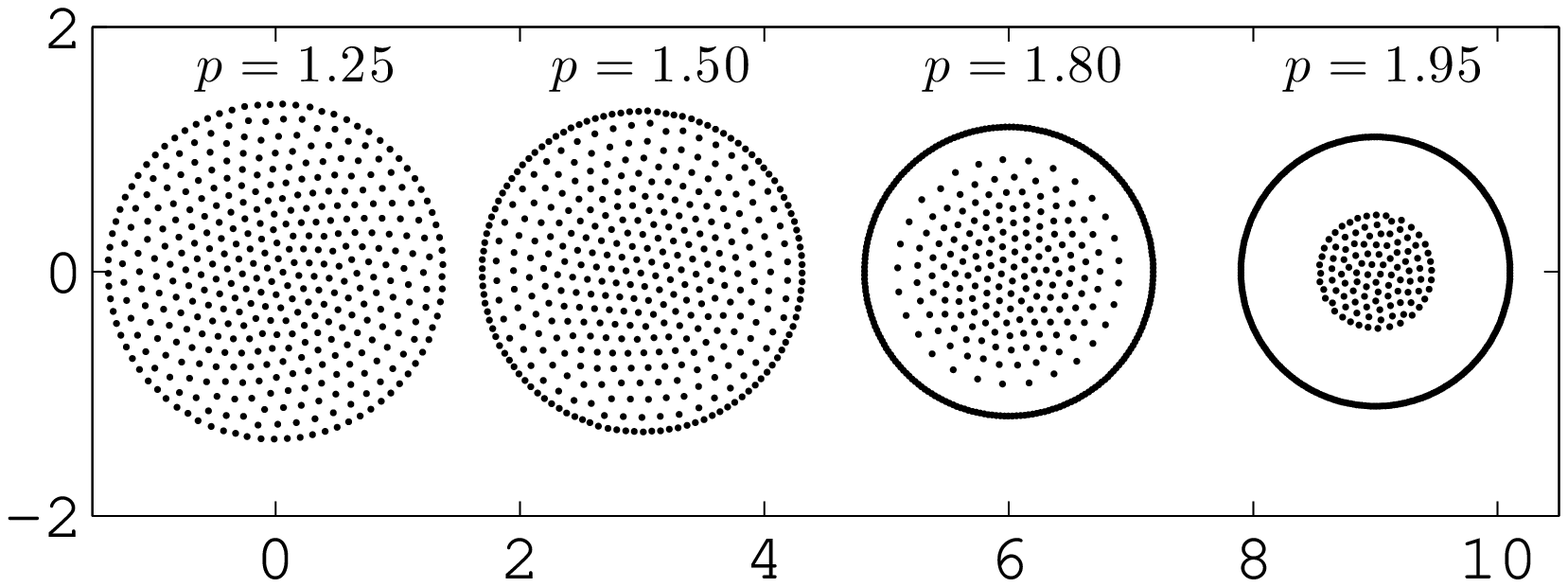}}
\end{center}
    \caption{The flock profiles from the particle
        simulations of the first-order
        system~\eqref{eq:firstorder} for the generalised
        Morse-like potential with $V(r)=-e^{r^p/p}$.}
        \label{fig:general}
\end{figure}

This concentration and dimensionality of the support of the steady
density is related to the singularity of $U$ near the origin, as
 has already been demonstrated in~\cite{BCLR2}. Here, we have to
argue by numerical experiments as existence proofs will be
difficult, partially because of the absence of explicit formulas. 
Similarly discussions can be found in~\cite{Kolokolnikov201365} 
for solutions perturbed from a ring solution, and in~\cite{MR2788924,MR3124884} for extensive 1D examples 
with $\delta$-concentration on a domain boundary. However, a detailed analytical 
investigation of these and other properties,
such as the integrability of the density near the boundary, remains a
challenging question for the potentials considered.


\section{Conclusion}
In this paper, we analyzed the solvability of convolution
equations that describe particular solutions in aggregation or
self-propelled interacting particle models equipped with radially
symmetric interaction potentials. Although models such as
\eqref{eq:partsys2} and \eqref{eq:firstorder} have been frequently
used with various potentials, the analysis of particular solutions
such as flock profiles and rotating mills is far from complete. We
concentrated our attention on the study of flock profiles, defined
as compactly supported continuous radial densities satisfying
equation \eqref{eq:steadyeq}. Focusing on the case of Quasi-Morse
potentials introduced in \cite{Carrillo2013}, we were able to
analytically study the parameter phase portrait of these
potentials in two and three dimensions, and to prove analytically
solvability conditions for flock profiles that were previously
asserted numerically. These findings are summarized in Figure
\ref{figresult}: The aggregate potential parameter $A$ determines
solvability in the biologically relevant parameter regimes. In
three dimensions, we showed existence and uniqueness of flock
profiles for $A>0$, whereas no flock profiles exists if $A\leq 0$.
The same non-existence result holds true in two dimensions, where
flock profiles are shown to exist if and only if $A>0$. The proof
of our main Theorems \ref{theo3d} and \ref{theo2d} is based on a
technical discussion of the Bessel functions contained in the
definition of the Quasi-Morse potentials and the explicit formulas
of their flock profiles obtained in \cite{Carrillo2013}. First, an
explicit expression for the convolution $W \star \rho$ was
derived for the three cases $A>0, A=0$ and $A<0$. Then, a detailed
analysis of the resulting expressions enabled us to establish our
theorems. A central observation is the fact that the question of
existence and uniqueness of flock profiles reduces to the study of
roots of a determinant of the coefficient matrix $M$. Due to the
simpler functions involved, results obtained in three dimensions
are slightly stronger than in two dimensions.

In summary, this paper is the first to our knowledge to complete a
full analytical study of the existence of flock profiles in the
biologically relevant parameter regime, at least for a particular
potential. The analysis of the Quasi-Morse potential and our
simulations seem to indicate the existence of flock solutions as
long as the potential has a unique positive global minimum and is
not H-stable. Characterizing when they are flock profiles is
challenging and related to the dimensionality of the support of
minimizers of the interaction energy \cite{BCLR2}. Proving or
disproving these claims for other potentials in
\eqref{eq:partsys2}, such as the Morse-type potentials
\eqref{morselike}, as well as the question of stability of such
states in the dynamics of the associated PDEs however remains an
open and challenging problem.

\appendix

\section{Bessel functions and Modified Bessel functions}
\label{sec:bessel} In this paper, Bessel functions and modified
Bessel functions are heavily used to study the
analytically more tractable Quasi-Morse type
potential~\eqref{eq:quasimorse}. The definitions and key
properties of these Bessel functions, found in standard textbook in special
functions~\cite{MR0232968}, are collected below for the readers'
convenience.

The Bessel functions of the first kind $J_\nu(x)$ and of the
second kind $Y_\nu(x)$ are solutions of the equation
\begin{equation}\label{eq:besselde}
x^2\frac{d^2y}{dx^2}+x\frac{dy}{dx}+(x^2-\nu^2)y=0,
\end{equation}
that are finite and singular at the origin for positive $\nu$,
respectively. The modified Bessel function of the first kind
$I_\nu(x)$ and of the second kind $K_\nu(x)$ are solutions of
the equation
\begin{equation}\label{eq:modifiedbesselde}
x^2\frac{d^2y}{dx^2}+x\frac{dy}{dx}-(x^2+\nu^2)y=0,
\end{equation}
that are exponentially growing and decaying, respectively.

In two and three dimensions considered in this paper, the
(modified) Bessel functions with negative order $\nu$ can
be rewritten
in terms of those with positive order. In particular,
in two dimensions we have
\begin{equation}\label{eq:2dbessel}
J_{-1}(x) = - J_{1}(x),\quad I_{-1}(x)=I_{1}(x),\quad
\quad K_{-1}(x)=K_{1}(x),
\end{equation}
and in three dimensions, we have the following explicit
representations using the well-known
(hyperbolic) trigonometric functions
\begin{subequations}\label{eq:3dbessel}
\begin{equation}\label{eq:3dbessela}
J_{1/2}(x) = \sqrt{\frac{2}{\pi x}} \sin x,\quad
J_{-1/2}(x)= \sqrt{\frac{2}{\pi x}} \cos x,\quad
\end{equation}
\begin{equation}\label{eq:3dbesselb}
K_{1/2}(x)=K_{-1/2}(x)=\sqrt{\frac{\pi }{2x}} e^{-x},
\end{equation}
\begin{equation}\label{eq:3dbesselc}
\quad I_{1/2}(x)=\sqrt{\frac{2}{\pi x}} \sinh x,\quad
I_{-1/2}(x)=\sqrt{\frac{2}{\pi x}} \cosh x.
\end{equation}
\end{subequations}

\textbf{Recursive relations.} In the proof of the
Lemma~\ref{lem:monbesk}, the following recursive relations for the
modified Bessel function $I_\nu(x)$ and $K_{\nu}(x)$ are used
\begin{subequations}
\begin{align}\label{eq:recBesselIK}
I_\nu'(x) = I_{\nu-1}(x)-\frac{\nu}{x}I_\nu(x),\qquad
I_\nu'(x) = \frac{\nu}{x}I_\nu(x)+I_{\nu+1}(x), \cr
K_\nu'(x) = -K_{\nu-1}(x)-\frac{\nu}{x}K_\nu(x),\qquad
K_\nu'(x) = \frac{\nu}{x}K_\nu(x)-K_{\nu+1}(x).
\end{align}
In the equivalent integral form, the following  are used to evaluate~\eqref{eq:repconvint}
and in the proof of Proposition \ref{prop1} in Appendix B,
\begin{equation}\label{eq:intBesselIK}
 \int x^\nu I_{\nu-1}(x)dx = x^\nu I_{\nu}(x),\qquad
 \int x^\nu K_{\nu-1}(x)dx = -x^\nu K_{\nu}(x).
\end{equation}
\end{subequations}

\textbf{Asymptotic expansions.} In the proof of the
Lemma~\ref{lem:monbesk}, the following asymptotic expansions of
$K_\nu(x)$ for $x>0$ are also needed. When $x>0$ is close to the
origin,
\begin{equation}\label{eq:bksmall}
    K_\nu(x) \approx \begin{cases}
        -\ln \frac{x}{2}-\gamma, \qquad & \nu = 0,\cr
        \Gamma(\nu) 2^{\nu-1} x^{-\nu}, & \nu>0,
    \end{cases}
\end{equation}
with the Euler constant $\gamma$. When $x$ is large,
\begin{equation}\label{eq:bklarge}
    K_\nu (x)= \left(\frac{2}{\pi x}\right)^{1/2}e^{-x}\left[
1+\frac{4\nu^2-1}{8x}+\frac{(4\nu^2-1)(4\nu^2-9)}{2!(8x)^2}+\cdots
\right] \approx     K_{1/2}(x).
\end{equation}

\textbf{Additional identities and integrals.} The most important
identity to simplify the final expressions
in~\eqref{eq:repconvint} and in the proof of Proposition
\ref{prop1} in Appendix B, is
\begin{equation}\label{eq:KIexchange}
 K_{\nu+1}(x)I_{\nu}(x)+
K_{\nu}(x)I_{\nu+1}(x)=\frac{1}{x}.
\end{equation}

Finally, we need the following integrals involving products
of two Bessel functions~\cite[p.~87]{MR0232968} to
evaluate~\eqref{eq:repconvint} ,
\begin{subequations}
\begin{align}
 \int xJ_{\nu}(ax)K_{\nu}\left(\frac{kx}{\ell}\right) dx
&= -\frac{\ell^2}{k^2+a^2\ell^2}\left[
axJ_{\nu-1}(ax)K_{\nu}\left(\frac{kx}{\ell}\right)
+\frac{kx}{\ell}J_{\nu}(ax)K_{\nu-1}\left(\frac{kx}{\ell}\right)
\right], \label{eq:BesselJKint} \\
\int xJ_{\nu}(ax)I_{\nu}\left(\frac{kx}{\ell}\right) dx
&= \frac{\ell^2}{k^2+a^2\ell^2}\left[
-axJ_{\nu-1}(ax)I_{\nu}\left(\frac{kx}{\ell}\right)
+\frac{kx}{\ell}J_{\nu}(ax)I_{\nu-1}\left(\frac{kx}{\ell}\right)
\right],\label{eq:BesselJIint}\\
 \int xI_{\nu}(ax)K_{\nu}\left(\frac{kx}{\ell}\right) dx
&= \frac{\ell^2}{a^2\ell^2-k^2}\left[
axI_{\nu-1}(ax)K_{\nu}\left(\frac{kx}{\ell}\right)
+\frac{kx}{\ell}I_{\nu}(ax)K_{\nu-1}\left(\frac{kx}{\ell}\right)
\right], \label{eq:BesselIKint}\\
\int xI_{\nu}(ax)I_{\nu}\left(\frac{kx}{\ell}\right) dx
&= \frac{\ell^2}{a^2\ell^2-k^2}\left[
axI_{\nu-1}(ax)I_{\nu}\left(\frac{kx}{\ell}\right)
-\frac{kx}{\ell}I_{\nu}(ax)I_{\nu-1}\left(\frac{kx}{\ell}\right)
\right].\label{eq:BesselIIint}
\end{align}
\end{subequations}


\section{Proof of Proposition \ref{prop1}}
\label{sec:expconvdev} Here, we focus on the integrals related to
$V_\ell$, because those related to $V$ are obtained by
evaluating at $C=1$ and $\ell=1$.

First, we evaluate the integral~\eqref{eq:repconvint} when
$\rho(s)$ are the linearly independent functions in the general
solution~\eqref{eq:explicitsol}, i.e., the constant $1$, $r^2$,
$r^{1-n/2}J_{n/2-1}(ar)$ and $r^{1-n/2}I_{n/2-1}(ar)$
respectively. When $\rho(s) =1$,
\begin{align*}
&\quad K_{\frac{n}{2}-1}(kr/\ell)
\int_0^r s^{\frac{n}{2}}I_{\frac{n}{2}-1}(ks/\ell)ds
+I_{\frac{n}{2}-1}(kr/\ell)
\int_r^R s^{\frac{n}{2}}K_{\frac{n}{2}-1}(ks/\ell)ds  \cr
&= \left. \frac{\ell}{k}K_{\frac{n}{2}-1}\left(\frac{kr}{\ell}\right)
s^{\frac{n}{2}}I_{\frac{n}{2}}\left(\frac{ks}{\ell}\right)\right|_{s=0}^r -
\left. \frac{\ell}{k}I_{\frac{n}{2}-1}\left(\frac{kr}{\ell}\right)
s^{\frac{n}{2}}K_{\frac{n}{2}}\left(\frac{ks}{\ell}\right)\right|_{s=r}^R
\qquad \Bigl( \text{by }~\eqref{eq:intBesselIK}\Bigr) \cr
&= \frac{\ell}{k}r^{\frac{n}{2}}\left[
K_{\frac{n}{2}-1}\left(\frac{kr}{\ell}\right)
I_{\frac{n}{2}}\left(\frac{kr}{\ell}\right)
+I_{\frac{n}{2}-1}\left(\frac{kr}{\ell}\right)
K_{\frac{n}{2}}\left(\frac{kr}{\ell}\right)
\right]-\frac{\ell}{k}R^{\frac{n}{2}}
I_{\frac{n}{2}-1}\left(\frac{kr}{\ell}\right)
K_{\frac{n}{2}}\left(\frac{kR}{\ell}\right)\cr
&= \frac{\ell^2}{k^2}r^{\frac{n}{2}-1}
-\frac{\ell}{k}R^{\frac{n}{2}}
I_{\frac{n}{2}-1}\left(\frac{kr}{\ell}\right)
K_{\frac{n}{2}}\left(\frac{kR}{\ell}\right). \qquad\qquad\qquad
\qquad\qquad \Bigl(\text{by }
\eqref{eq:KIexchange}\Bigr)
\end{align*}
When $\rho(s)=r^2$, using~\eqref{eq:intBesselIK} and
integration by parts, we get
\begin{align*}
 \int s^{\frac{n}{2}+2} K_{\frac{n}{2}-1}(ks/\ell) ds &=
-\frac{\ell}{k}s^{\frac{n}{2}+2}K_{\frac{n}{2}}\left(
\frac{ks}{\ell}\right)-\frac{2\ell^2}{k^2}
s^{\frac{n}{2}+1}K_{\frac{n}{2}+1}\left(
\frac{ks}{\ell}\right), \cr
 \int s^{\frac{n}{2}+2} I_{\frac{n}{2}-1}(ks/\ell) ds &=
\frac{\ell}{k}s^{\frac{n}{2}+2}I_{\frac{n}{2}}\left(
\frac{ks}{\ell}\right)-\frac{2\ell^2}{k^2}
s^{\frac{n}{2}+1}I_{\frac{n}{2}+1}\left(
\frac{ks}{\ell}\right),
\end{align*}
and hence
\begin{subequations}
\begin{align}
 &\quad K_{\frac{n}{2}-1}(kr/\ell)
\int_0^r s^{\frac{n}{2}+2}I_{\frac{n}{2}-1}(ks/\ell)ds
+I_{\frac{n}{2}-1}(kr/\ell)
\int_r^R s^{\frac{n}{2}+2}K_{\frac{n}{2}-1}(ks/\ell)ds  \cr
&=\frac{\ell}{k}r^{\frac{n}{2}+2}\left[
K_{\frac{n}{2}-1}\left(\frac{kr}{\ell}\right)
I_{\frac{n}{2}}\left(\frac{kr}{\ell}\right)
+I_{\frac{n}{2}-1}\left(\frac{kr}{\ell}\right)
K_{\frac{n}{2}}\left(\frac{kr}{\ell}\right)
\right] \label{eq:rsqsimp1}\\
&\qquad +\frac{2\ell^2}{k^2}r^{\frac{n}{2}+1}\left[
I_{\frac{n}{2}-1}\left(\frac{kr}{\ell}\right)
K_{\frac{n}{2}+1}\left(\frac{kr}{\ell}\right)
-K_{\frac{n}{2}-1}\left(\frac{kr}{\ell}\right)
I_{\frac{n}{2}+1}\left(\frac{kr}{\ell}\right)
\right] \label{eq:rsqsimp2}\\
&\qquad - \left[\frac{\ell}{k}R^{\frac{n}{2}+2}K_{\frac{n}{2}}
\left(\frac{kR}{\ell}\right)
+\frac{2\ell^2}{k^2}R^{\frac{n}{2}+1}
K_{\frac{n}{2}+1}
\left(\frac{kR}{\ell}\right)
\right]I_{\frac{n}{2}-1}\left(\frac{kr}{\ell}\right) \cr
&=\frac{\ell^2}{k^2}r^{\frac{n}{2}+1}
+\frac{2\ell^4}{k^4}r^{\frac{n}{2}-1}
- R^{\frac{n}{2}+1}
\left[\frac{\ell}{k}RK_{\frac{n}{2}}
\left(\frac{kR}{\ell}\right)
+\frac{2\ell^2}{k^2}
K_{\frac{n}{2}+1}
\left(\frac{kR}{\ell}\right)
\right]I_{\frac{n}{2}-1}\left(\frac{kr}{\ell}\right). \notag
\end{align}
\end{subequations}
Here the terms inside the square bracket of~\eqref{eq:rsqsimp1}
or ~\eqref{eq:rsqsimp2} are equal to $\ell/kr$ or
$2n^2\ell^2/(kr)^2$,
by the
recursive relations~\eqref{eq:recBesselIK} and the
identity~\eqref{eq:KIexchange}.

When $\rho(s) = s^{n/2-1}J_{n/2-1}(as)$, using~\eqref{eq:BesselJKint}
and~\eqref{eq:BesselJIint},
\begin{align*}
 &\quad  K_{\frac{n}{2}-1}(kr/\ell)
\int_0^r sI_{\frac{n}{2}-1}(ks/\ell)J_{\frac{n}{2}-1}(as)ds
+I_{\frac{n}{2}-1}(kr/\ell)
\int_r^R sK_{\frac{n}{2}-1}(ks/\ell)J_{\frac{n}{2}-1}(ar)ds
\cr
&=\frac{rk\ell}{a^2\ell^2+k^2}\left[
I_{\frac{n}{2}-1}\left(\frac{kr}{\ell}\right)
K_{\frac{n}{2}-2}\left(\frac{kr}{\ell}\right)+
I_{\frac{n}{2}-2}\left(\frac{kr}{\ell}\right)
K_{\frac{n}{2}-1}\left(\frac{kr}{\ell}\right)
\right]J_{\frac{n}{2}-1}(ar)
\cr &\quad -\frac{R\ell}{a^2\ell^2+k^2}\left[
kJ_{\frac{n}{2}-1}(aR)K_{\frac{n}{2}-2}\left(\frac{kR}{\ell}\right)
+a\ell J_{\frac{n}{2}-2}(aR)K_{\frac{n}{2}-1}\left(\frac{kR}{\ell}\right)
\right] I_{\frac{n}{2}-1}\left(\frac{kr}{\ell}\right)
\cr &=
\frac{\ell^2}{a^2\ell^2+k^2}J_{\frac{n}{2}-1}(ar)
-\frac{R\ell}{a^2\ell^2+k^2}\left[
kJ_{\frac{n}{2}-1}(aR)K_{\frac{n}{2}-2}\left(\frac{kR}{\ell}\right)
\right.
\cr &\left.\qquad\qquad
+a\ell J_{\frac{n}{2}-2}(aR)K_{\frac{n}{2}-1}\left(\frac{kR}{\ell}\right)
\right]
I_{\frac{n}{2}-1}\left(\frac{kr}{\ell}\right).
\end{align*}

Finally when $\rho(s)=s^{n/2-1}I_{n/2-1}(as)$,
using~\eqref{eq:BesselIKint} and~\eqref{eq:BesselIIint},
\begin{align*}
&\quad K_{\frac{n}{2}-1}(kr/\ell)
\int_0^r sI_{\frac{n}{2}-1}(ks/\ell)I_{\frac{n}{2}-1}(as)ds
+I_{\frac{n}{2}-1}(kr/\ell)
\int_r^R sK_{\frac{n}{2}-1}(ks/\ell)I_{\frac{n}{2}-1}(ar)ds \cr
&=-\frac{rk\ell}{a^2\ell^2-k^2}\left[
I_{\frac{n}{2}-1}\left(\frac{kr}{\ell}\right)
K_{\frac{n}{2}-2}\left(\frac{kr}{\ell}\right)+
I_{\frac{n}{2}-2}\left(\frac{kr}{\ell}\right)
K_{\frac{n}{2}-1}\left(\frac{kr}{\ell}\right)
\right]I_{\frac{n}{2}-1}(ar) \cr
&\quad +\frac{R\ell}{a^2\ell^2-k^2}\left[
kI_{\frac{n}{2}-1}(aR)K_{\frac{n}{2}-2}\left(\frac{kR}{\ell}\right)
+a\ell I_{\frac{n}{2}-2}(aR)K_{\frac{n}{2}-1}\left(\frac{kR}{\ell}\right)
\right] \cr
&= -\frac{\ell^2}{a^2\ell^2+k^2}J_{\frac{n}{2}-1}(ar)
+\frac{R\ell}{a^2\ell^2-k^2}\left[
kI_{\frac{n}{2}-1}(aR)K_{\frac{n}{2}-2}\left(\frac{kR}{\ell}\right)
\right. \cr
&\left. \qquad\qquad
+a\ell I_{\frac{n}{2}-2}(aR)K_{\frac{n}{2}-1}\left(\frac{kR}{\ell}\right)
\right] I_{\frac{n}{2}-1}\left(\frac{kr}{\ell}\right).
\end{align*}
Putting all the integrals together, we conclude the explicit
form~\eqref{eq:explicitconv} for the convolution $W\star \rho$.
For example, when $A>0$, $\rho(r) = \mu_1 r^{1-\frac{n}{2}} J_{
\frac{n}{2}-1}(ar)+\mu_2$, collecting the terms in the
integral~\eqref{eq:angint}, we get
\begin{align*}
    (W\star \rho)(r)  &= \mu_2 \frac{C\ell^n-1}{k^2}
    + \mu_1 r^{1-\frac{n}{2}}\left(\frac{C\ell^n}{a^2\ell^2+k^2}
    -\frac{1}{a^2+k^2}\right)J_{\frac{n}{2}-1}(ar) \cr
    &    -
    r^{1-\frac{n}{2}}
    \left\{ \mu_1\frac{C\ell^{n-1}R}{a^2\ell^2+k^2}\left[
            kJ_{\frac{n}{2}-1}(aR)K_{\frac{n}{2}-2}\left(\frac{kR}{\ell}
            \right)+a\ell J_{\frac{n}{2}-2}(aR)K_{\frac{n}{2}-1}\left(
            \frac{kR}{\ell}\right)\right]       \right. \cr
            &\left. + \mu_2 \frac{\ell}{k}K_{\frac{n}{2}}\left(
            \frac{kR}{\ell}\right)
        \right\}
    I_{\frac{n}{2}-1}\left(\frac{kr}{\ell}\right)
    +r^{1-\frac{n}{2}}\left\{ \mu_1 \frac{R}{a^2+k^2}
    \left[ kJ_{\frac{n}{2}-1}(aR)K_{\frac{n}{2}-2}(kR)
        \right.\right.\cr
&\left.\left.
+aJ_{\frac{n}{2}-2}(aR)K_{\frac{n}{2}-1}(kR)+\mu_2 \frac{1}{k}
K_{\frac{n}{2}}(kR)
\right]
    \right\} I_{\frac{n}{2}-1}(kr)\,.
\end{align*}
The first term $\mu_2 (C\ell^n-1)/k^2$ is the desired constant
$D$, and the
factor $C\ell^n/(a^2\ell^2+k^2)-1/(a^2+k^2)$ in the second term
vanishes by the definition of $a^2$. The rest of the terms are a
linear combination of $I_{\frac{n}{2}-1}(kr/\ell)$ and
$I_{\frac{n}{2}-1}(kr)$, and they can be rearranged into the
form~\eqref{eq:explicitsol} with the coefficient of $\mu_2$
normalized to one to simplify the later proofs. The explicit form
for $W\star \rho$ when $A=0$ or $A<0$ has similar structures,
and its simplification leads to the final expression~\eqref{eq:explicitconv}.

\subsection*{Acknowledgments}
JAC was supported by projects MTM2011-27739-C04-02 and
2009-SGR-345 from Ag\`encia de Gesti\'o d'Ajuts Universitaris i de
Recerca-Generalitat de Catalunya. JAC acknowledges support from
the Royal Society through a Wolfson Research Merit Award. JAC, YH,
and SM were supported by Engineering and Physical Sciences
Research Council grant number EP/K008404/1.

\end{document}